\documentclass{birkjour}
\usepackage{hyperref}
\usepackage{microtype}
\usepackage{amsfonts, tikz}
\usepackage{amssymb,amsmath}
\usepackage{amsthm,mathtools}
\usepackage{graphicx}
\usepackage{amscd}
\newtheorem{theo}{Theorem}[section]
\newtheorem{cor}[theo]{Corollary}
\newtheorem{lem}[theo]{Lemma}
\newtheorem{prop}[theo]{Proposition}
\theoremstyle{definition}
\newtheorem{defn}[theo]{Definition}
\theoremstyle{remark}
\newtheorem{rmk}[theo]{Remark}
\newtheorem{egg}[theo]{Example}
\numberwithin{equation}{section}

\renewcommand{\labelenumi}{(\roman{enumi})}

\newcommand{\Hom}{\mathrm{Hom}}
\newcommand{\id}{\mathrm{id}}
\newcommand{\B}{\mathcal{B}}

\newcommand{\I}{\mathbf{I}}

\DeclareMathOperator{\diag}{diag}

\begin{document}
		\title[\(C^*\)-diagonal of inductive limit of \(1\)-dimensional NCCW complexes]{\(C^*\)-diagonal of inductive limit of \(1\)-dimensional NCCW complexes}
		\author{Dolapo Oyetunbi}
		
		\address{Department of Mathematics and Statistics\\
		University of Windsor\\
		Canada}
		\email{Oyetunb@uwindsor.ca}
		
		\begin{abstract}This paper establishes the existence of a $C^*$-diagonal in the inductive limit of 1-dimensional NCCW complexes with trivial $K_1$-groups. It also examines some limitations and implications of approximating $^*$-homomorphisms between two such complexes.
		\end{abstract}
		\maketitle
	\section{Introduction}
1-dimensional noncommutative CW-complexes (NCCW complexes) are a crucial class of \(C^*\)-algebras that extend $1$-dimensional classical CW-complexes into the non-commutative world. Key characteristics of these \(C^*\)-algebras include being finitely generated \cite[Lemma 2.3]{Eilers98} and semiprojective \cite[Lemma 2.3]{Eilers98}, which simplifies their study and analysis.
Due to their well-understood structure, they serve as useful building blocks in the classification of more complex \(C^*\)-algebras.


Renault and Kumjian\cite{kum86, Ren08} introduced the concept of \( C^* \)-diagonals within \( C^* \)-algebras, following the foundational work of Feldman and Moore \cite{FM77} on Cartan subalgebras of von Neumann algebras. They showed that \( C^* \)-algebras with \( C^* \)-diagonals are realizable as \( C^* \)-algebras of well-behaved topological groupoids.

A significant amount of work has been devoted to identifying \( C^* \)-diagonals within \( C^* \)-algebras. For instance, an important result by an Huef et al. \cite{AKS11} shows that the unique extension property of an abelian \( C^* \)-subalgebra of a Fell algebra—and in particular, of a \( C^* \)-algebra with continuous trace—is both a necessary and sufficient condition for the subalgebra to be a \( C^* \)-diagonal. Recently, Li and Raad \cite{Li+Raad23} established the existence of \( C^* \)-diagonals in AH-algebras, which include certain non-simple and non-classifiable \( C^* \)-algebras. 
We generalize certain special cases from the framework of Li and Raad and establish the existence of \(C^*\)-diagonals in new classes of C-algebras, including some non-simple C*-algebras, which lie beyond the scope of their results to prove the following:

\begin{theo}[see Theorem \ref{thm:Mainresult}]\label{thm:Main}
	Every unital inductive limit of one-dimensional NCCW complexes with trivial \( K_1 \)-group and unital injective connecting maps has a \( C^* \)-diagonal.
\end{theo}
While classifying real rank zero \( C^* \)-algebras, Elliott \cite[Theorem 4.4]{Elliott93} showed that maximally homogeneous \(^*\)-ho\-mo\-mor\-phisms between matrix algebras over \( C(S^1) \) are dense in the space of unital \(^*\)-homomorphisms between these \( C^* \)-algebras when equipped with the point-norm topology. Li \cite[Theorem 2.1.6]{Li97} extended this result to the case of matrix algebras over one-dimensional connected CW-complexes. Li and Raad \cite{Li+Raad23} relied on these results to prove that unital AH-algebras whose building blocks have base spaces of dimension at most one admit \( C^* \)-diagonals. 

We examine the extension of maximal homogeneity to \(^*\)-homomorphisms between one-dimensional NCCW complexes and highlight both the limited supply of such maps and the limitations of this approach for establishing the existence of \( C^* \)-diagonals. To establish Theorem \ref{thm:Main}, we introduce the notion of an \( n \)-standard map between one-dimensional NCCW complexes, prove that any unital \(^*\)-homomorphism between two such complexes can be approximated by an \( n \)-standard map, and combine ideas from Robert's classification \cite{Robert12} of inductive limits of one-dimensional NCCW complexes with Li's results \cite{Li20} on \( C^* \)-diagonals. As a consequence, we obtain:
\begin{cor}[see Corollary \ref{cor:InductiveLimit}]
	Let $A \coloneq \varinjlim (A_n, \phi_n)$ be a unital, separable $C^*$-algebra, where each $A_n$ is a 1-dimensional NCCW complex with trivial $K_1$-groups, and the connecting maps $\phi_n$ are unital and injective. Then $A$ can be written as an inductive limit of 1-dimensional NCCW complexes with unital, injective connecting maps, each of which is $n$-standard for some $n$.
\end{cor}

\vspace*{1cm}
\noindent \textbf{Acknowledgments}: This work forms part of the author’s Ph.D. thesis completed at the University of Ottawa under the supervision of Aaron Tikuisis. The author expresses deep gratitude to Aaron Tikuisis for his invaluable guidance, insightful discussions, and unwavering support throughout the Ph.D. journey.

\section{Preliminaries}
Based on \cite{Eilers98}, we define a 1-dimensional NCCW complex as follows:
\begin{defn} A \emph{1-dimensional noncommutative CW-complex (or simply 1-NCCW complex)} is a \(C^*\)-algebra  $A\coloneqq A(E,F,\beta_0 , \beta_1)$  defined as follows: Let $\beta_0 , \beta_1 : E\rightarrow F$ be $^*$-homomorphisms between two finite-dimensional $C^*$-algebras $E$ and $F$. The $C^*$-algebra $A\coloneqq A(E,F,\beta_0 , \beta_1)$ is the pullback of $C([0,1], F)$ and $E$ along $\beta \coloneq \beta_0 \oplus \beta_1 : E\rightarrow F\oplus F$ and $ev_0 \oplus ev_1 : C([0,1], F) \rightarrow F\oplus F$. 
	\begin{center}

		\tikzset{every picture/.style={line width=0.75pt}} 
		
		\begin{tikzpicture}[x=0.75pt,y=0.75pt,yscale=-1,xscale=1]
			
			\draw    (291.2,98.6) -- (354.2,97.63) ;
			\draw [shift={(356.2,97.6)}, rotate = 179.12] [color={rgb, 255:red, 0; green, 0; blue, 0 }  ][line width=0.75]    (8.74,-2.63) .. controls (5.56,-1.12) and (2.65,-0.24) .. (0,0) .. controls (2.65,0.24) and (5.56,1.12) .. (8.74,2.63)   ;
			\draw    (280.09,113.64) -- (280.2,160.6) ;
			\draw [shift={(280.2,162.6)}, rotate = 269.87] [color={rgb, 255:red, 0; green, 0; blue, 0 }  ][line width=0.75]    (8.74,-2.63) .. controls (5.56,-1.12) and (2.65,-0.24) .. (0,0) .. controls (2.65,0.24) and (5.56,1.12) .. (8.74,2.63)   ;
			\draw    (290,178.6) -- (383,178.6) ;
			\draw [shift={(385,178.6)}, rotate = 180] [color={rgb, 255:red, 0; green, 0; blue, 0 }  ][line width=0.75]    (8.74,-2.63) .. controls (5.56,-1.12) and (2.65,-0.24) .. (0,0) .. controls (2.65,0.24) and (5.56,1.12) .. (8.74,2.63)   ;
			\draw    (407.09,113.64) -- (407.2,160.6) ;
			\draw [shift={(407.2,162.6)}, rotate = 269.87] [color={rgb, 255:red, 0; green, 0; blue, 0 }  ][line width=0.75]    (8.74,-2.63) .. controls (5.56,-1.12) and (2.65,-0.24) .. (0,0) .. controls (2.65,0.24) and (5.56,1.12) .. (8.74,2.63)   ;
			
			\draw (122,147) node [anchor=north west][inner sep=0.75pt]   [align=left] {$ $};
			\draw (360.98,89.48) node [anchor=north west][inner sep=0.75pt]    {$C([ 0,1] ,F)$};
			\draw (271.64,169.68) node [anchor=north west][inner sep=0.75pt]    {$E$};
			\draw (391.27,170.92) node [anchor=north west][inner sep=0.75pt]    {$F\oplus F$};
			\draw (271.96,88.97) node [anchor=north west][inner sep=0.75pt]    {$A$};
			\draw (340,188.4) node [anchor=north west][inner sep=0.75pt]    {$\beta $};
			\draw (412,124.4) node [anchor=north west][inner sep=0.75pt]    {$ev_{0} \oplus ev_{1}$};
			\draw (256,126.4) node [anchor=north west][inner sep=0.75pt]    {$\pi _{2}$};
			\draw (317,74.4) node [anchor=north west][inner sep=0.75pt]    {$\pi _{1}$};

		\end{tikzpicture}
	\end{center}
	We can express $A(E,F,\beta_0 ,\beta_1)$ explicitly as
	\begin{equation}
		\{(f,a)\in C([0,1],F)\oplus E: f(0)=\beta_0(a), f(1)=\beta_{1}(a)\}.
	\end{equation}
\end{defn}
Note that \( A(E,F,\beta_0,\beta_1) \) is a unital \(C^*\)-algebra if and only if \(\beta_0\) and \(\beta_1\) are unital \(^*\)-homomorphisms. We denote the set of all unital 1-dimensional NCCW complexes and finite-dimensional \(C^*\)-algebras by \(\text{1-NCCW}_1\) (up to isomorphism).

For \( A \coloneq A(E,F,\beta_0, \beta_1) \in \text{1-NCCW}_1 \), we write \( E \) and \( F \) as follows:
\begin{equation}\label{DefnOfEandF}
	E \coloneq \bigoplus_{i=1}^{l} M_{e_i}(\mathbb{C}) \quad \text{and} \quad F \coloneq \bigoplus_{j=1}^{k} M_{f_j}(\mathbb{C}),
\end{equation}
where \( e_i \) and \( f_j \) are natural numbers. 

Throughout what follows, we will assume \( A \coloneq A(E, F, \beta_0, \beta_1) \in \text{1-NCCW}_{1} \) satisfies conditions (A1) and (A2) below : 
\begin{enumerate}
	\item[(A1)] \(\beta \coloneq \beta_0 \oplus \beta_1 : E \rightarrow F \oplus F\) is injective.
	
	If \(\beta\) is not injective, take \(E' = E / \ker(\beta)\) and \(\beta'_0, \beta'_1 : E' \rightarrow F\) as \(^*\)-homomorphisms induced by \(\beta_0\) and \(\beta_1\). Then \(A' \coloneq A(E', F, \beta'_0, \beta'_1) \in \text{1-NCCW}_1\) with injective \(^*\)-homomorphism \(\beta' : E' \rightarrow F \oplus F\) induced by \(\beta\). Moreover, \(A \cong A' \oplus \ker(\beta)\). Note that the injectivity of either \(\beta_0\) or \(\beta_1\) implies the injectivity of \(\beta\). However, the converse is not true in general. Using Condition (A1), we can identify \(A\) with the \(C^*\)-subalgebra 
	\begin{equation}\label{InjectiveIdentification}
		\{f \in C([0,1],F) : (f(0), f(1)) \in \operatorname{Im}(\beta)\}
	\end{equation} 
	of \(C([0,1], F)\). We will often identify \(A\) in this manner without explicitly stating it whenever we assume (A1) holds.
	
	\item[(A2)] \(\beta_0 : E \rightarrow F\) and \(\beta_1 : E \rightarrow F\) are of the following form:
	\begin{equation}\label{eqn:A2}
		\beta_0(a_1, a_2, \ldots, a_l) = u \left(\bigoplus_{i=1}^{l} (a_i \otimes \I_{r_i})\right) u^*,
	\end{equation}
	and
	\begin{equation}\label{eqn:A2b}
		\beta_1(a_1, a_2, \ldots, a_l) = v \left(\bigoplus_{i=1}^{l} (a_i \otimes \I_{s_i})\right) v^*,
	\end{equation}
	where \(u\) and \(v\) are permutation matrices in \(F\), \(r_i\) and \(s_i\) are nonnegative integers, \(\I_{r_i}\) is the identity matrix of size \(r_i\), and \(\sum\limits_{i=1}^{l} e_i r_i = \sum\limits_{j=1}^{k} f_j = \sum\limits_{i=1}^{l} e_i s_i\).
\end{enumerate}
Up to isomorphism of \(C^*\)-algebras, every \(\beta_0\) and \(\beta_1\) of \(A \in \text{1-NCCW}_1\) satisfy condition (A2) with \(u\) and \(v\) both equal to the identity matrix. To see this, we first note that there exist unitaries \(U, V \in F\) such that
\begin{equation}\label{eqn:A2general}
	\beta_0(a) = U \left(\bigoplus_{i=1}^{l} (a_i \otimes \I_{r_i})\right) U^* \quad \text{and} \quad \beta_1(a) = V \left(\bigoplus_{i=1}^{l} (a_i \otimes \I_{s_i})\right) V^*
\end{equation}
for \(a = (a_1, a_2, \ldots, a_l) \in E\), since \(\beta_0\) and \(\beta_1\) are unital.Let \( w \in C([0,1], \mathcal{U}(F)) \) be a continuous path of unitaries satisfying \( w(0) = U^* \) and \( w(1) = V^* \). Define \( B \coloneq A(E,F,\beta_0', \beta_1') \), where \(\beta_0' = U^* \beta_0 U\) and \(\beta_1' = V^* \beta_1 V\). Then, \(\beta_0'\) and \(\beta_1'\) of \( B \in \text{1-NCCW}_1 \) satisfy condition (A2), and \( A \cong B \) via \((f, a) \mapsto (w f w^*, a)\). Note that we leave \( u \) and \( v \) as permutation matrices in equations \eqref{eqn:A2} and \eqref{eqn:A2b} to accommodate future constructions of new 1-NCCW complexes.

Examples of 1-dimensional NCCW complexes include \( C(X, M_n) \), where \( X \) is a 1-dimensional connected finite CW-complex, the splitting interval algebra \cite{Santiago10}, and the dimension-drop algebra \cite{Jiang-Su99}. Both the prime dimension-drop algebra and \( C(X, M_n) \), where \( X \) is a contractible, connected, 1-dimensional finite CW-complex, have trivial \( K_1 \)-groups (see \cite{Robert12}).

A finite multiset, denoted by \(\{a_1^{\sim k_1}, \ldots, a_m^{\sim k_m}\}\), extends the concept of a standard finite set by allowing elements to appear multiple times, with \(a_i^{\sim k_i}\) representing \(\underbrace{a_i, \ldots, a_i}_{k_i \,\text{times}}\). 

Suppose \(A \coloneq A(E,F,\beta_0, \beta_1) \in \text{1-NCCW}_1\), with \(E\) and \(F\) defined as in \eqref{DefnOfEandF}. We represent the irreducible representation of \(E\) that projects \(E\) onto its \(i\)-th component \(M_{e_i}(\mathbb{C})\) by \(\delta_i\) for \(1 \leq i \leq l\), and denote the set of unitary equivalence classes of all the \(\delta_i\) by \(Sp(E)\).	Similarly, we use the notation \((t,i)\), where \(0 \leq t \leq 1\) and \(1 \leq i \leq k\), to denote an irreducible representation of \(C([0,1], F)\) that sends \(f\) to the \(i\)-th component of \(f(t)\). Therefore, the set of unitary equivalence classes of all irreducible representations of \(C([0,1], F)\) is given by
\[
Sp(C([0,1], F)) \coloneq \coprod_{i=1}^{k} \{(t,i) : 0 \leq t \leq 1\}.
\]
For \((f,a) \in A\), the condition \(f(0) = \beta_0(a)\) and the form of \(\beta_0\) as in \eqref{eqn:A2} indicate that \((0,i)\) may not correspond to an irreducible representation of \(A\). Instead, it can be identified with \(\{\delta_1^{\sim r_{i1}}, \delta_2^{\sim r_{i2}}, \ldots, \delta_l^{\sim r_{il}}\} \subset Sp(E)\). Similarly, \((1,i)\) can be identified with \(\{\delta_1^{\sim s_{i1}}, \delta_2^{\sim s_{i2}}, \ldots, \delta_l^{\sim s_{il}}\} \subset Sp(E)\). So, the set of unitary equivalence classes of irreducible representations of \(A\) is given by
\[
Sp(A) \coloneq Sp(E) \cup \coprod_{i=1}^{k}\{(t,i) : 0 < t < 1\} = Sp(E) \cup \coprod_{i=1}^{k}(0,1)_i.
\]
We write \(sp(A)\coloneq Sp(E) \cup \coprod\limits_{i=1}^{k}[0,1]_i\), and further adopt the following notations and definitions from \cite{Liu19}:

\( \pi_t : A \rightarrow F \) is defined by \( \pi_t(f,a) = f(t) \) for all \( t \in [0,1] \).  \( \pi_e : A \rightarrow E \) and \( q : A \rightarrow F \) are defined by \( \pi_e(f,a) = a \) and \( q(f,a) = f \). Let \( (\alpha_{ij})_{k \times l} \) and \( (\beta_{ij})_{k \times l} \) be matrices representing \( (\beta_0)^* : K_0(E) = \mathbb{Z}^l \rightarrow K_0(F) = \mathbb{Z}^k \) and \( (\beta_1)^* : K_0(E) = \mathbb{Z}^l \rightarrow K_0(F) = \mathbb{Z}^k \), respectively. We use \( \bullet \) or \( \bullet\bullet \) to represent any nonnegative integer.

The topology on $\mathrm{Sp}(A)$ has a basis consisting of open sets of the form
\[
\{ \delta_j \} \cup \coprod_{\{ i \mid \alpha_{ij} \neq 0 \}} (0, \epsilon)_i \cup \coprod_{\{ i \mid \beta_{ij} \neq 0 \}} (1 - \epsilon, 1)_i
\]
for each $0 < \epsilon < 1$ and each $j$, and of the form $(a, b)_i$ for each $0 \leq a < b \leq 1$ and each $i$. This defines a topology that is non-Hausdorff in general.  Moreover, for each \( j \) such that \( \alpha_{ij} \neq 0 \), \( (t, i) \in \operatorname{Sp}(A) \) converges to \( \delta_j \) as \( t \to 0 \); similarly, for each \( j \) such that \( \beta_{ij} \neq 0 \), the \( (t, i) \) converges to \( \delta_j \) as \( t \to 1 \).

In what follows, we recall the construction of test functions useful for relating properties of two \( ^* \)-homomorphisms from a 1-dimensional NCCW complex to a matrix algebra.

For each \(\eta = \frac{1}{m}\), where \(m\) is a positive integer, consider a partition \(0 = w_0 < w_1 < \cdots < w_m = 1\) of \([0,1]\) into \(m\) subintervals, each of equal length \(\frac{1}{m}\). Let \(H(\eta)\) be a finite subset of \(A\) that consists of the following two types of positive elements of \(A\):
\begin{enumerate}
	\item[(a)] For  \(1 \leq i \leq k\) and any nonnegative integers \(a_i\) and \(b_i\) satisfying \(0 \leq a_i < a_i + 2 \le b_i \leq m\), a test function \((f,a)\) of type 1 corresponding to \(\{\delta_j\} \cup \coprod\limits_{\{i \mid \alpha_{ij} \neq 0\}} [0, a_i \eta]_i \cup \coprod\limits_{\{i \mid \beta_{ij} \neq 0\}} [b_i \eta, 1]_i\) is defined by
	\begin{equation}\label{eqn:Type1a}
		a(\delta_s) \coloneq \begin{cases}
			0_{e_s} & \text{if } s \neq j \\
			\mathbf{I}_{e_j} & \text{if } s = j
		\end{cases},
	\end{equation}
	where \(0_{e_s}\) is the zero matrix and \(\mathbf{I}_{e_j}\) is the identity matrix, and
	\begin{equation}\label{eqn:Type1b}
		f(t,i) \coloneq \begin{cases}
			\beta_0^i(a) \frac{\eta - \text{dist}(t, [0, a_i \eta]_i)}{\eta} & \text{if } 0 \leq t \leq (a_i + 1) \eta \\
			0_{f_i} & \text{if } (a_i + 1) \eta \leq t \leq (b_i - 1) \eta \\
			\beta_1^i(a) \frac{\eta - \text{dist}(t, [b_i \eta, 1]_i)}{\eta} & \text{if } (b_i - 1) \eta \leq t \leq 1
		\end{cases}
	\end{equation}
	for each \(t \in [0,1]\) and \(i \in \{1, 2, \ldots, k\}\), where \(\beta_{s}^{i}(a)\) is the \(i\)-th component of \(\beta_{s}(a)\) and \(s = 0, 1\).		
	
	\item[(b)] A test function \((f,a)\) of type 2 corresponding to \(X \coloneq \bigcup\limits_s [w_{r_s}, w_{r_{s+1}}]_i \subset [\eta, 1 - \eta]_i\) is defined by
	\begin{equation}\label{eqn:Type2a}
		a \coloneq 0
	\end{equation}
	and
	\begin{equation}\label{eqn:Type2b}
		f(t, r) \coloneq \begin{cases}
			0_{f_r} & \text{if } r \neq i \\
			(1 - \frac{\text{dist}(t, X)}{\eta}) \mathbf{I}_{f_i} & \text{if } \text{dist}(t, X) < \eta \text{ and } r = i \\
			0_{f_i} & \text{if } \text{dist}(t, X) \geq \eta \text{ and } r = i
		\end{cases}.
	\end{equation}
\end{enumerate}
Let \(\{e_{m m'}^i: 1 \leq m, m' \leq e_i,\,\, 1 \leq i \leq l\}\) and \(\{f_{m m'}^i: 1 \leq m, m' \leq f_i, \,\, 1 \leq i \leq k\}\) be the standard matrix units of \(E\) and \(F\), respectively. We modify elements of \(H(\eta)\) to define a new finite set \(\widetilde{H}(\eta) \subset A_+\). \(\widetilde{H}(\eta)\) consists of elements of the form:
\begin{itemize}
	\item \((f,a)\) constructed like test functions of type 1 with \(\mathbf{I}_{e_j}\) in \eqref{eqn:Type1a} replaced by \(e_{m m'}^j\) for any \(m, m'\).
	\item \((f,a)\) constructed like test functions of type 2 with \((1 - \frac{\text{dist}(t,X)}{\eta}) \mathbf{I}_{f_i}\) in \eqref{eqn:Type2b} replaced by \((1 - \frac{\text{dist}(t,X)}{\eta}) e_{m m'}^i\) for any \(m, m'\).
\end{itemize}
Next, we recall a calculation of the Cuntz semigroup of a 1-dimensional NCCW complex (see \cite{APS11} for a brief overview of the Cuntz semigroup).
\begin{theo}[{\cite[Theorem 3.1]{APS11}}]\label{thm:CuntzCalculation}
	Let \(A \coloneq A(E,F, \beta_0 , \beta_1)\) be a 1-dimensional NCCW complex. Define \(\gamma_0 \coloneq Cu(\beta_0) : Cu(E) \rightarrow Cu(F)\) and \(\gamma_1 \coloneq Cu(\beta_1) : Cu(E) \rightarrow Cu(F)\). Then,
	\begin{equation}
		Cu(A) \cong \{(f,b) \in Lsc([0,1], \overline{\mathbb{N}}^k) \oplus \overline{\mathbb{N}}^l : f(0) = \gamma_0(b), f(1) = \gamma_1(b)\},
	\end{equation}
	where \(\overline{\mathbb{N}} \coloneq \mathbb{N} \cup \{\infty\}\) and \( \mathrm{Lsc}([0,1], \overline{\mathbb{N}}^k) \) denotes the set of lower semicontinuous functions from the interval \([0,1]\) to \(\overline{\mathbb{N}}^k\).
\end{theo}
\begin{defn}
	Let \(\B\) be a \(C^*\)-subalgebra of a \(C^*\)-algebra \(\mathcal{A}\). We say \(\B\) is a Cartan subalgebra of \(\mathcal{A}\) or \((\B, \mathcal{A})\) is a Cartan pair if
	\begin{enumerate}
		\item \(\B\) is a maximal abelian \(C^*\)-subalgebra of \(\mathcal{A}\),
		\item \(\B\) is regular in \(\mathcal{A}\), that is, \(\mathcal{A}\) is generated by the set of normalizers
		\[
		N_{\mathcal{A}}(\B) = \{ n \in \mathcal{A} : n\B n^*, n^* \B n \subseteq \B \},
		\]
		\item There exists a faithful conditional expectation \(P: \mathcal{A} \rightarrow \B\), that is, \(P(a^*a)= 0\) implies \(a=0\).
	\end{enumerate}
\end{defn}
\begin{defn}\label{defn:Diagonal}
	Let \(\B\) be a \(C^*\)-subalgebra of a \(C^*\)-algebra \(\mathcal{A}\). We say \(\B\) is a \(C^*\)-diagonal of \(\mathcal{A}\) if \(\B\) is a Cartan subalgebra with the unique extension property; that is, every pure state of \(\B\) extends uniquely to a pure state of \(\mathcal{A}\).
\end{defn}

\begin{egg}
	\label{egg:ExampleCartanPair}
	\((D_n(\mathbb{C}), M_n(\mathbb{C}))\) is a Cartan pair, where \(D_n(\mathbb{C})\) consists of all \(n \times n\) diagonal matrices. Indeed, \(D_n(\mathbb{C})\) is a maximal abelian subalgebra since it has vector space dimension \(n\), and every commutative subalgebra of \(M_n(\mathbb{C})\) has dimension at most \(n\). The standard matrix units that generate \(M_n(\mathbb{C})\) are normalizers of \(D_n(\mathbb{C})\), and the canonical conditional expectation from \(M_n(\mathbb{C})\) onto \(D_n(\mathbb{C})\) is faithful. It can be further shown that \(D_n(\mathbb{C})\) is a \(C^*\)-diagonal.
\end{egg}

Following the regularity idea in \cite[Proposition 5.1]{BS21}, one can prove the following:
\begin{prop} \label{prop:C*-diagonal}Let $A \coloneq A(E,F,\beta_0, \beta_1) \in\text{1-NCCW}_1$ and $B \coloneq \{ (f,a) \in A \mid f(t) \in \bigoplus \limits_{i=1}^{k} D_{f_{i}}(\mathbb{C}) \text{ for all } t \in [0,1] \}$. Then $B$ is a $C^*$-diagonal of $A$.
\end{prop}
\section{Maximally Homogeneity and Standard Maps}
We begin by generalizing the notion of maximally homogeneous \( ^* \)-homomorphisms between matrix algebras over 1-dimensional CW-complexes, as introduced by \cite{Elliott93}, to 1-dimensional NCCW complexes. We then examine some limitations associated with this type of morphism for 1-dimensional NCCW complexes.

Let \( A \coloneq A(E, F, \beta_0, \beta_1) \in \text{1-NCCW}_{1} \) and let \( \phi: A \rightarrow M_n (\mathbb{C}) \) be a \( ^* \)-homomorphism. Then there exists a unitary \( u \in M_n (\mathbb{C}) \) such that
\begin{equation} \label{eqn:GeneralSpectrum}
	\phi(f, a) = u  \diag \left( a(\delta_1) \otimes \mathbf{I}_{s_1}, \dots, a(\delta_l) \otimes \mathbf{I}_{s_l}, f(w_1), f(w_2), \ldots, f(w_r), 0_\bullet \right) u^*,
\end{equation}
where \( w_1, \ldots, w_r \in \coprod\limits_{i=1}^{k} (0,1)_i \). We define the spectrum of \( \phi \) by
\begin{equation}
	\text{Sp}(\phi) \coloneq \{ \delta_1^{\sim s_1}, \delta_2^{\sim s_2}, \ldots, \delta_l^{\sim s_l}, w_1, w_2, \ldots, w_r \}.
\end{equation}

Let \( (X, d) \) be a metric space. We say two finite sets \( X_1 = \{ w_1, w_2, \ldots, w_n \} \subset X \) and \( X_1' = \{ w_1', w_2', \ldots, w_m' \} \) can be bijectively paired within \( \eta > 0 \) if \( m = n \) and there exists a permutation \( \sigma \) such that \( d(w_i, w_{\sigma(i)}') < \eta \) for all \( 1 \le i \le n \). While it is not always possible to pair the spectra of any two \( ^* \)-homomorphisms from \( A \in \text{1-NCCW}_{1} \) to \( M_n (\mathbb{C}) \) within a given \( \eta = \frac{1}{m} \), where \( m \) is a positive integer, we can pair some subsets of their spectra under some mild constraints.
\begin{lem}[{\cite[Lemma 2.3]{Liu19b}}]\label{lem:Pairing}
	Let \( A \coloneq A(E, F, \beta_0, \beta_1) \in \text{1-NCCW}_{1} \), \( m \) be any positive integer, \( \epsilon \) be any positive real number less than 1, and let \( \phi, \psi: A \rightarrow M_n (\mathbb{C}) \) be \( ^* \)-homomorphisms. If \( \text{Eig}(\phi(h)) \) \footnote{Set of eigenvalues of \( \phi(h) \)} and \( \text{Eig}(\psi(h)) \) can be bijectively paired within \( \epsilon \) for all \( h \in H(\eta) \), where \( \eta \coloneq \frac{1}{m} \), then for each \( i \in \{1, 2, \ldots, k\} \), there exist \( X_i \subset \text{Sp}(\phi) \cap (0,1)_i \) and \( X_i' \subset \text{Sp}(\psi) \cap (0,1)_i \) satisfying the following conditions:
	\begin{enumerate}
		\item \( X_i \) and \( X_i' \) can be bijectively paired within \( 2\eta \),
		\item \( \text{Sp}(\phi) \cap [\eta, 1-\eta]_i \subset X_i \) and \( \text{Sp}(\psi) \cap [\eta, 1-\eta]_i \subset X_i' \).
	\end{enumerate}
\end{lem}
By the Weyl spectral variation inequality \cite{Bhatia07}, \( \Vert \phi(h) - \psi(h) \Vert < \epsilon \) for all \( h \in H(\eta) \) is sufficient to establish the bijective pairing of \( \text{Eig}(\phi(h)) \) and \( \text{Eig}(\psi(h)) \) within \( \epsilon \).

Following \cite[Definition 2.2]{Li+Raad23}, we extend the concept of maximally homogeneous \( ^* \)-homomorphisms to maps between two \( C^* \)-algebras \( A, B \in \text{1-NCCW}_{1} \).
\begin{defn}
	Let \( A, B \in \text{1-NCCW}_{1} \). A unital \( ^* \)-homomorphism \( \phi: A \rightarrow B \) is called maximally homogeneous if, for any \( \mu \in \text{Sp}(B) \), every element of \( \text{Sp}(\phi_\mu) \) has multiplicity at most one, where \( \phi_\mu \coloneqq \mu \circ \phi \).
\end{defn}

\begin{egg}\label{ex:NCCW1}
	Let \( A \coloneq A(E, F, \beta_0, \beta_1) \in \text{1-NCCW}_{1} \) with \( E \coloneq \mathbb{C}^3 \), \( F \coloneq M_2(\mathbb{C}) \), and 
	\begin{equation}
		\beta_0 (a_1, a_2, a_3) \coloneq \begin{pmatrix} 
			a_1 & 0 \\
			0 & a_3
		\end{pmatrix} \quad \text{and} \quad \beta_1 (a_1, a_2, a_3) \coloneq \begin{pmatrix} 
			a_2 & 0 \\
			0 & a_3
		\end{pmatrix}.
	\end{equation}
	The \( ^* \)-homomorphism \( \phi: A \rightarrow C([0,1], M_3(\mathbb{C})) \) defined by
	\begin{equation}
		\phi_t(f, (a_1, a_2, a_3)) \coloneq \begin{cases}
			\begin{pmatrix} 
				a_1 & 0 \\
				0 & f\left(t + \frac{1}{2}\right)
			\end{pmatrix} & \text{if } t \in [0, \frac{1}{2}] \\
			\begin{pmatrix}
				1 & 0 & 0 \\
				0 & 0 & 1 \\
				0 & 1 & 0
			\end{pmatrix} 
			\begin{pmatrix}
				f\left(t - \frac{1}{2}\right) & 0 \\
				0 & a_2
			\end{pmatrix}
			\begin{pmatrix}
				1 & 0 & 0 \\
				0 & 0 & 1 \\
				0 & 1 & 0
			\end{pmatrix} & \text{if } t \in [\frac{1}{2}, 1]
		\end{cases}
	\end{equation}
	is a maximally homogeneous \( ^* \)-homomorphism.
\end{egg}
Let \( A \coloneq A(E, F, \beta_0, \beta_1) \in \text{1-NCCW}_{1} \) and \( \phi: A \rightarrow M_n(\mathbb{C}) \) be a unital \( ^* \)-homomorphism. By grouping the \( \delta_i \)'s and using \eqref{eqn:A2} and \eqref{eqn:A2b}, we can reformulate \eqref{eqn:GeneralSpectrum} to have a unitary \( v \in M_n(\mathbb{C}) \) and
\begin{equation}\label{eqn:SpectrumForm2}
	\phi(f, a) = v \diag\left( a(\delta_1) \otimes \mathbf{I}_{s_1^{'}}, \dots, a(\delta_l) \otimes \mathbf{I}_{s_l^{'}}, f(w_1^{'}) , f(w_2^{'}) , \ldots, f(w_{r^{'}}^{'}) \right) v^*
\end{equation}
with \( w_1^{'}, \ldots, w_{r^{'}}^{'} \in \coprod\limits_{i=1}^{k} [0,1]_i \) and \( 0 \leq s_i^{'} \leq s_i \). 


We use an idea of \cite[Lemma 3.5]{Liu19} to prove the following:
\begin{prop}
Let \( A \coloneqq A(E,F,\beta_0,\beta_1) \in \text{1-NCCW}_1 \), and let \( \phi: A \to M_n(\mathbb{C}) \) be a unital \( ^* \)-homomorphism. Then \( \phi \) is a limit of maximally homogeneous \( ^* \)-homomorphisms if and only if there exists a formulation of \( \phi \) as in \eqref{eqn:SpectrumForm2} where the multiplicity at each \( \delta_i \) is at most one.
\end{prop}
\begin{proof}
Suppose \( \phi \) is a limit of maximally homogeneous \( ^* \)-homomorphisms. Express \( \phi \) in the form given by \eqref{eqn:SpectrumForm2}, and choose a positive integer \( m \) such that, for all \( 1 \le i \le r' \), each nontrivial \( w_i' \) lies in \( \coprod\limits_{j=1}^{k} \left[ \frac{1}{m}, 1 - \frac{1}{m} \right]_j \). Define \( \eta \coloneqq \frac{1}{8mn} \).
Then, there exists a maximally homogeneous \( ^* \)-homomorphism \( \psi: A \rightarrow M_n (\mathbb{C}) \) such that \( \Vert \phi(h) - \psi(h) \Vert < 1 \) for all \( h \in H(\eta) \). By Lemma \ref{lem:Pairing}, there exists \( [\eta, 1-\eta]_i \subset X_i \subset \text{Sp}(\psi) \cap (0,1)_i \) for each \( i \in \{1,2,\ldots, k\} \) such that \( \text{Sp}(\phi) \cap (0,1)_i \) and \( X_i \) can be bijectively paired within \( 2\eta \). Let \( \varpi_i :X_i \rightarrow \text{Sp}(\phi) \cap (0,1)_i \) be the map that defines the bijective pairing. Define a new \( ^* \)-homomorphism \( \psi_1 : A \rightarrow M_n (\mathbb{C}) \) from \( \psi \) by changing all \( w \in \left((0,\eta)_i \cap \text{Sp}(\psi)\right) \setminus X_i \) to \( 0_i \), all \( w \in \left((1-\eta, 1)_i \cap \text{Sp}(\psi)\right) \setminus X_i \) to \( 1_i \), and all \( w \in X_i \) to \( \varpi_i(w) \). Note that, by construction, \( \psi_1 \) can be written in the form given by \eqref{eqn:SpectrumForm2}, where each \( \delta_i \) appears with multiplicity at most one. Moreover,

\begin{equation}\label{eqn:Newspectrum=Old}
	\text{Sp}(\phi) \cap (0,1)_i = \text{Sp}(\psi_1) \cap (0,1)_i
\end{equation}
for each \( i = 1,2,\ldots,k \). For each \( [0,1]_i \), choose positive integers \( a_i, b_i \) such that \( 1 < a_i < a_i + 2 \le b_i < 8mn \) and

\[ \text{Sp}(\phi) \cap (a_i\eta, b_i\eta)_i = \text{Sp}(\psi) \cap (a_i\eta, b_i\eta)_i = \emptyset .\]

For a test function \( h_j \) of type 1 corresponding to \( \{\delta_j\} \cup \coprod\limits_{\{i \vert \alpha_{ij}\neq 0\}}[0,a_i\eta]_i \cup \coprod\limits_{\{i \vert \beta_{ij}\neq 0\}}[b_i\eta,1]_i \), we have that \( \phi(h_j) \), \( \psi(h_j) \) are projections, and
\begin{equation}\label{eqn:TestFunctionUse1}
	\psi(h_j) = \psi_1(h_j) \quad \text{and} \quad \Vert \phi(h_j) - \psi(h_j) \Vert < 1
\end{equation}
for each \( j = 1,2,\ldots, k \). \eqref{eqn:Newspectrum=Old} and the inequality in \eqref{eqn:TestFunctionUse1} ensure \( \text{Sp}(\phi) \cap \text{Sp}(E) = \text{Sp}(\psi_1) \cap \text{Sp}(E) \). 
Hence, the result holds.

Conversely, consider \( \phi \) in the form of \eqref{eqn:SpectrumForm2} and suppose \( s_{i}^{'} = 0 \,\text{or}\, 1 \) for all \( 0 \le i \le k \). Given a finite set \( \mathcal{F} \) and \( \epsilon > 0 \), there exists a \( \delta > 0 \) such that \( \text{dist}(x,y) < \delta \) implies \( \Vert f(x) - f(y) \Vert < \epsilon \) for all \( (f,a) \in \mathcal{F} \). A \( ^* \)-homomorphism \( \psi: A \rightarrow M_n (\mathbb{C}) \) can be defined by substituting points in \( \text{Sp}(\phi) \cap \coprod\limits_{i=1}^{k} [0,1]_i \) with distinct non-zero points in \( \coprod\limits_{i=1}^{k} (0,1)_i \) within \( \delta \). This \( \psi \) is maximally homogeneous and satisfies \( \Vert \phi(h) - \psi(h) \Vert < \epsilon \) for all \( h \in \mathcal{F} \).
\end{proof}
\begin{egg}
Let \( A \) be as in Example \ref{ex:NCCW1}. The \( ^* \)-homomorphism \( \phi_0 : A \rightarrow M_3(\mathbb{C}) \) defined by \( \phi_{0}(f,(a_1, a_2, a_3)) \coloneq \diag (a_1, f(1)) \) is homotopic to the \( ^* \)-homomorphism \( \phi_1 : A \rightarrow M_3(\mathbb{C}) \) defined by \( \phi_{1}(f,(a_1, a_2, a_3)) \coloneq \diag (a_1, f(0)) \). Notice that \( \phi_0 \) is maximally homogeneous, while \( \phi_1 \) is not. Therefore, homotopy to a maximally homogeneous \( ^* \)-homomorphism does not guarantee maximal homogeneity.
\end{egg}

Unlike unital \( ^* \)-homomorphisms between matrix algebras over 1-dim\-en\-sional CW-complexes, not every unital \( ^* \)-homomorphism between two \( C^* \)-algebras \( A, B \in \text{1-NCCW}_1 \) can be approximated by a maximally homogeneous \( ^* \)-homomorphism. In fact, there may not exist any maximally homogeneous \( ^* \)-homomorphism between two 1-dimensional NCCW complexes, as the next example highlights.

\begin{egg}
Recall the dimension drop algebra
\begin{equation}
	\mathcal{Z}_{p,q} \coloneq \{f \in C([0,1], M_{pq} (\mathbb{C})) : f(0) \in M_p (\mathbb{C}) \otimes \mathbf{I}_q \,\,\text{and}\,\, f(1) \in \mathbf{I}_p \otimes M_q (\mathbb{C})\}
\end{equation}
is a 1-dimensional NCCW complex with \( E \coloneq M_p \oplus M_q \), \( F \coloneq M_{pq} \), \( \beta_0 (a,b) \coloneq a \otimes \mathbf{I}_q \), and \( \beta_1 (a,b) \coloneq b \otimes \mathbf{I}_p \).

\( \Hom (\mathcal{Z}_{2,3}, \mathcal{Z}_{2,5}) \) is non-empty since it contains a \( ^* \)-homomorphism \( \phi: \mathcal{Z}_{2,3} \rightarrow \mathcal{Z}_{2,5} \) defined by
\begin{equation}
	\phi_t (f, (a,b)) \coloneq  u_t \cdot \diag(a \otimes \mathbf{I}_2, f(t)) \cdot u_t^*,
\end{equation}
where \( u_t \in C([0,1], M_{10}(\mathbb{C})) \) defines a continuous path of unitaries from the identity matrix to the permutation matrix
\[
\begin{pmatrix}
	\mathbf{I}_2 & 0_{2 \times 2} & 0_{2 \times 3} & 0_{2 \times 3} \\
	0_{3 \times 2} & 0_{3 \times 2} & \mathbf{I}_3 & 0_{3 \times 3} \\
	0_{2 \times 2} & \mathbf{I}_2 & 0_{2 \times 3} & 0_{2 \times 3} \\
	0_{3 \times 2} & 0_{3 \times 2} & 0_{3 \times 3} & \mathbf{I}_3
\end{pmatrix}.
\]
However, there is no maximally homogeneous \( ^* \)-homomorphism from \( \mathcal{Z}_{2,3} \) to \( \mathcal{Z}_{2,5} \) due to the boundary conditions of \( \mathcal{Z}_{2,3} \) and \( \mathcal{Z}_{2,5} \). Indeed, suppose there is a maximally homogeneous \(^*\)-homomorphism \(\psi: \mathcal{Z}_{2,3} \rightarrow \mathcal{Z}_{2,5}\). For any \(t \in (0,1)\), suppose \(\psi_t : \mathcal{Z}_{2,3} \rightarrow M_{10}(\mathbb{C})\) is defined by  
\begin{equation}
	\label{eqn:psi-mu}
	\psi_t(f, a) = v \, \diag(\delta_1 \otimes \I_{s_1}, \, \delta_2 \otimes \I_{s_2}, \, f(w_1), \, f(w_2), \, \ldots, \, f(w_r)) \, v
\end{equation}  
for all \((f, a) \in \mathcal{Z}_{2,3}\), where \(v\) and \(w_i\) are as in \eqref{eqn:GeneralSpectrum}. Here,  
\[
\delta_1 : M_2(\mathbb{C}) \oplus M_3(\mathbb{C}) \rightarrow M_2(\mathbb{C}) \quad \text{and} \quad
\delta_2 : M_2(\mathbb{C}) \oplus M_3(\mathbb{C}) \rightarrow M_3(\mathbb{C})
\]  
are defined by \(\delta_i(a_1, a_2) = a_i\) for \(i = 1, 2\).

Then \(r \leq 1\), since each \(f(t)\) lies in \(M_6(\mathbb{C})\), and the total matrix size is fixed at \(10\). If \(r = 1\), then the two summands \(\delta_1 \otimes \I_{s_1}\) and \(\delta_2 \otimes \I_{s_2}\) must together account for dimension \(4\), which forces \(s_1 = 2\). But this contradicts the maximal homogeneity of \(\psi\).

On the other hand, if \(r = 0\), then both \(\delta_1\) and \(\delta_2\) must occur with multiplicity greater than one in order to define a unital \(^*\)-homomorphism into \(M_{10}(\mathbb{C})\). This again contradicts maximal homogeneity.

Therefore, no such maximally homogeneous \(^*\)-homomorphism \(\psi: \mathcal{Z}_{2,3} \rightarrow \mathcal{Z}_{2,5}\) can exist.
\end{egg}
The following example demonstrates that the approximate maximal homogeneity of each \( \phi_\mu \) for a unital \( ^* \)-homomorphism \( \phi \) between 1-dimensional NCCW complexes does not ensure the approximate maximal homogeneity of the \( ^* \)-homomorphism itself.
\begin{egg}
Let \( A \coloneq A(E,F,\beta_0,\beta_1) \) with \( E \coloneq \mathbb{C}^3 \), \( F \coloneq M_3(\mathbb{C}) \), and
\begin{equation}
	\beta_0(a_1, a_2, a_3) \coloneq \diag(a_1, a_2, a_3), \quad \beta_1(a_1, a_2, a_3) \coloneq \diag(a_1, a_1, a_2).
\end{equation}		
\\
\resizebox{0.9\linewidth}{!}{
	\begin{minipage}{\linewidth}
		Define \( \phi : A \rightarrow C([0,1], M_4(\mathbb{C})) \) by
		\begin{equation*}
			\phi_t(f, (a_1, a_2, a_3)) \coloneq \begin{cases}
				\begin{pmatrix}
					a_1 & 0 \\
					0 & f\left(\frac{1}{2} - t\right)
				\end{pmatrix} & \text{if } t \in [0, \frac{1}{2}], \\
				\begin{pmatrix}
					0 & 0 & 0 & 1 \\
					1 & 0 & 0 & 0 \\
					0 & 1 & 0 & 0 \\
					0 & 0 & 1 & 0
				\end{pmatrix}
				\begin{pmatrix}
					a_3 & 0 \\
					0 & f\left(\frac{3}{2} - t\right)
				\end{pmatrix}
				\begin{pmatrix}
					0 & 1 & 0 & 0 \\
					0 & 0 & 1 & 0 \\
					0 & 0 & 0 & 1 \\
					1 & 0 & 0 & 0
				\end{pmatrix} & \text{if } t \in [\frac{1}{2}, 1].
			\end{cases}
		\end{equation*}
	\end{minipage}
}\\

Then every \( \phi_t \) is a maximally homogeneous \( ^* \)-homomorphism except for \( \phi_{\frac{1}{2}} \), which is only approximately maximally homogeneous. Suppose \( \phi \) is the limit of maximally homogeneous \( ^* \)-homomorphisms. Choose a fixed finite set \( \mathcal{F} \) and \( \delta > 0 \). Then, there exists a maximally homogeneous \( ^* \)-homomorphism \( \psi: A \rightarrow C([0,1], M_4(\mathbb{C})) \) satisfying \( \Vert \phi(h) - \psi(h) \Vert < \delta \) for all \( h \in \mathcal{F} \).

It follows that
\begin{equation}\label{eqn:homogeneitycounterexample}
	\Vert \phi_0(h) - \psi_0(h) \Vert < \delta \quad \text{and} \quad \Vert \phi_1(h) - \psi_1(h) \Vert < \delta
\end{equation}
for all \( h \in \mathcal{F} \).

By choosing \( \delta \) sufficiently small if necessary, we can assume \( \psi_0 \) and \( \psi_1 \) (using \eqref{eqn:homogeneitycounterexample}) are of the same form as \( \phi_0 \) and \( \phi_1 \) with variations in the evaluation points of \( f \). Moreover, the maximal homogeneity of \(\psi\) implies that for each \( t \in (0,1) \),
\[
\psi_t \left(f, (a_1, a_2, a_3)\right) = U 	\begin{pmatrix}
	a_i & 0 \\
	0 & f(w)
\end{pmatrix} U^*
\]
for some unitary matrix \( U \) and some \( w \in (0,1) \). Set \( g \coloneq (\text{diag}(0,0,1-t), (0,0,1)) \in A \).

Define
\[
S = \{\psi_t(g) : t \in [0,1]\},
\]
\[
S_1 = \{\psi_t(g) : t \in [0,1]\} \cap \{ V \text{diag}(1,0,0, 1-t)V^* : V \in \mathcal{U}(M_4(\mathbb{C})), 0 < t < 1 \},
\]
and
\[
S_2 \coloneq \{\psi_t(g) : t \in [0,1]\} \cap \{ V \text{diag}(0,0,0, 1-t)V^* : V \in \mathcal{U}(M_4(\mathbb{C})), 0 < t < 1 \}.
\]

Then, \( S = S_1 \cup S_2 \). Notice that \( S_1 \cap S_2 = \emptyset \) since matrices in \( S_1 \) always have eigenvalue 1 while matrices in \( S_2 \) never have such an eigenvalue. Therefore, \( S = \{\psi_t(g) : t \in [0,1]\} \) is not connected, and \( \psi \) does not exist. Hence, \(\phi\) is not a limit of maximally homogeneous \(^*\)-homomorphisms.
\end{egg}
Given the aforementioned limitations of maximally homogeneous \( ^* \)-homo\-morphisms between 1-dimensional NCCW complexes, it is evident that these maps cannot play the same role as maximally homogeneous \( ^* \)-homomorphisms between matrix algebras over 1-dimensional CW complexes in the classification of inductive limit \( C^* \)-algebras \cite{Li97, Elliott93} and inductive limit \( C^*\)-diagonals \cite{Li+Raad23}. Thomsen \cite{Thomsen87} showed that maximally homogeneous \( ^* \)-homomorphisms between matrix algebras over 1-dimensional CW complexes are standard maps. We will now present a variation of these standard maps for 1-dimensional NCCW complexes. This variation differs slightly from the one discussed in \cite{AEL20}, generalizes the one considered in \cite{Thomsen87} and is the same as the one implied in the work of \cite{Liu19}.
\begin{defn}\label{defn:n-standardform}
Let \( A \in \text{1-NCCW}_1 \) and \( \phi: A \rightarrow C([0,1], M_q(\mathbb{C})) \) be a *-homomorphism. We call \(\phi\) an \(n\)-standard map if there exists an \(n\)-partition \(I_1 = [z_0, z_1], I_2 = [z_1, z_2], \ldots, I_n = [z_{n-1}, z_n]\) of \([0,1]\) with \(z_0 = 0\) and \(z_n = 1\), and unitaries \(u^{(m)} \in C(I_m, M_q(\mathbb{C}))\) such that for all \(h \in A\) and \(t \in I_m\),
\begin{equation}\label{eqn:n-standardmap}
	\phi_t(h) \coloneqq u^{(m)}(t) \begin{pmatrix}
		h \circ \xi_1^{(m)}(t) & & & \\
		& h \circ \xi_2^{(m)}(t) & & \\
		& & \ddots & \\
		& & & h \circ \xi_{k(m)}^{(m)}(t) \\
		& & & & 0_\bullet
	\end{pmatrix} u^{(m)*}(t),
\end{equation}
where \(\xi_i^{(m)} \in C(I_m, \text{sp}(A))\) and \(m = 1, 2, \ldots, n\).
\end{defn}
\begin{rmk} Recall that if $h\coloneq (f,a)\in A$, then $h(\delta_i)\coloneq a(\delta_i)$ and $h(t,i)\coloneq f(t,i)$. Moreover, the number and nature of the eigenfunctions $\xi_i ^{(m)}$ depend on the subinterval $I_m$ and may not be the same for all $m=1,2,\ldots,n$. 
\end{rmk}
\begin{defn} Let \( A\) and \( B \coloneqq B(E', F', \alpha_0, \alpha_1) \in \text{1-NCCW}_1 \). A $^*$-homo\-morphism \(\phi: A \rightarrow B \subset C([0,1], F')\) is an \(n\)-standard map if for each \(i = 1, 2, \ldots, k'\), the map
	\begin{equation}
		q_i \circ \phi : A \rightarrow C([0,1], M_{f_i'}(\mathbb{C}))
	\end{equation}
	is an \(n\)-standard map, where
	\begin{equation}
		E' \coloneqq \bigoplus_{i=1}^{l'} M_{e_i'}(\mathbb{C}), \quad F' \coloneqq \bigoplus_{j=1}^{k'} M_{f_j'}(\mathbb{C}),
	\end{equation}
	and \( q_i: B \rightarrow C([0,1], M_{f_i'}(\mathbb{C})) \) is the natural projection map.
\end{defn}
\begin{egg}
	The $^*$-homomorphism \(\phi\) in Example \ref{ex:NCCW1} is a \(2\)-standard map with \(U: [0,1] \rightarrow M_3(\mathbb{C})\) defined by
	\[
	U(t) \coloneqq
	\begin{cases}
		\begin{pmatrix}
			1 & 0 & 0 \\
			0 & 1 & 0 \\
			0 & 0 & 1
		\end{pmatrix} & \text{if } t \in [0, \frac{1}{2}], \\
		\begin{pmatrix}
			1 & 0 & 0 \\
			0 & 0 & 1 \\
			0 & 1 & 0
		\end{pmatrix} & \text{if } t \in (\frac{1}{2}, 1].
	\end{cases}
	\]
\end{egg}
The next lemma shows the existence of a \(3\)-standard map connecting two close unital $^*$-ho\-mo\-mor\-phisms.
\begin{lem}[{\cite[Lemma 3.5]{Liu19}}]\label{lem:3-standardform} Let \( A \coloneqq A(E,F,\beta_0, \beta_1) \in \text{1-NCCW}_1 \), \(\mathcal{F} \subset A\) be a finite set, and \(\epsilon > 0\). There exist \(\eta\), \(\eta_1\), and \(\delta > 0\) such that if \(\phi_0, \phi_1 : A \rightarrow M_n (\mathbb{C})\) are unital $^*$-homomorphisms satisfying:
	\begin{enumerate}
		\item \(\Vert \phi_0(h) - \phi_1(h) \Vert < 1\) for all \(h \in H(\eta_1)\),
		\item \(\Vert \phi_0(h) - \phi_1(h) \Vert < \frac{\delta}{8}\) for all \(h \in H(\eta) \cup \widetilde{H}(\eta)\),
	\end{enumerate}
	then there exists a unital \(3\)-standard map \(\phi: A \rightarrow C([0,1], M_n(\mathbb{C}))\) connecting \(\phi_0\) to \(\phi_1\) such that
	\begin{equation}
		\Vert \phi_t(f) - \phi_0(f) \Vert < \epsilon \quad \forall f \in \mathcal{F}, \, t \in [0,1].
	\end{equation}
	Moreover, for each \(\mu \in (\text{Sp}(\phi_0) \cup \text{Sp}(\phi_1)) \cap \coprod\limits_{i=1}^{k}(0,1)_i\), we have 
	\[
	\overline{B_{4\eta_1}(\mu)} \subset \bigcup_{t \in [0,1]} \text{Sp}(\phi_t),
	\]
	where \(\overline{B_{4\eta_1}(\mu)} \coloneqq \{\upsilon \in \coprod\limits_{i=1}^{k}[0,1]_i : \text{dist}(\upsilon, \mu) \leq 4\eta_1 \}.\)
\end{lem}
\begin{proof}
	The full details are in \cite[Lemma 3.5]{Liu19}. Here, we present the form of the $3$-standard map, as it is needed for what follows.
	
	Take \(\delta\), \(\eta\), \(\eta_1\), and \( \text{Sp}(\phi_0) \cap [\eta_1, 1-\eta_1]_i \subset X_i \subset \text{Sp}(\phi_0) \cap (0,1)_i\), \( \text{Sp}(\phi_1) \cap [\eta_1, 1-\eta_1]_i \subset X_i^{'} \subset \text{Sp}(\phi_1) \cap (0,1)_i\) such that \(X_i\) and \(X_i^{'}\) can be bijectively paired within \(2\eta_1\) for \(i = 1, 2, \ldots, k\) (see \cite[Lemma 3.5]{Liu19}). Denote the one-to-one correspondence by \(\Lambda_i : X_i \rightarrow X_i^{'}\).
	
	There exist unitaries \(U_0, U_1 \in M_n(\mathbb{C})\) such that
	\begin{equation}\label{eqn:InitialPath}
		\phi_0(f,a) = U_0 \diag\left(a(\delta_1) \otimes \mathbf{I}_{s_1}, \dots, a(\delta_l) \otimes \mathbf{I}_{s_l}, f(w_1), f(w_2), \ldots, f(w_r)\right) U_0^*
	\end{equation}
	and
	\begin{equation}\label{eqn:FinalPath}
		\phi_1(f,a) = U_1 \cdot \diag(a(\delta_1) \otimes \mathbf{I}_{t_1}, \dots, a(\delta_l) \otimes \mathbf{I}_{t_l}, f(c_1), f(c_2), \ldots, f(c_q))U_1^*,
	\end{equation}
	with \(w_m, c_j \in \coprod\limits_{i=1}^{k} (0,1)_i\).
	
	For each \( w_m \in \text{Sp}(\phi_0) \cap (0,1)_i \), define a continuous map \(\xi_m^{(1)}: [0, \frac{1}{3}] \rightarrow \coprod\limits_{i=1}^{k} [0,1]_i\) satisfying:
	\begin{enumerate}
		\item \(\xi_m^{(1)}(0) \coloneqq w_m\),
		\item \(\xi_m^{(1)}\left(\frac{1}{3}\right) \coloneqq \begin{cases}
			0_i & \text{if } w_m \in (0, \eta_1)_i \setminus X_i, \\
			\Lambda(w_m) & \text{if } w_m \in X_i, \\
			1_i & \text{if } w_m \in (1 - \eta_1, 1)_i \setminus X_i,
		\end{cases}\)
		\item \(\text{Im}(\xi_m^{(1)}) = \overline{B_{4\eta_1}(w_m)} \coloneqq \{ w \in \coprod\limits_{i=1}^{k} [0,1]_i \mid \text{dist}(w, w_m) < 4\eta_1 \}\).
	\end{enumerate}
	Similarly, for each \( c_j \in \text{Sp}(\phi_0) \cap (0,1)_i \), define a continuous map \(\xi_j^{(3)}: [\frac{2}{3}, 1] \rightarrow \coprod\limits_{i=1}^{k} [0,1]_i\) satisfying:
	\begin{enumerate}
		\item \(\xi_j^{(3)}(0) \coloneqq c_j\),
		\item \(\xi_j^{(3)}\left(\frac{2}{3}\right) \coloneqq \begin{cases}
			0_i & \text{if } c_j \in (0, \eta_1)_i \setminus X_i, \\
			c_j & \text{if } c_j \in X_i, \\
			1_i & \text{if } c_j \in (1 - \eta_1, 1)_i \setminus X_i,
		\end{cases}\)
		\item \(\text{Im}(\xi_j^{(3)}) = \overline{B_{4\eta_1}(c_j)} \coloneqq \{ c \in \coprod\limits_{i=1}^{k} [0,1]_i \mid \text{dist}(c, c_j) < 4\eta_1 \}\).
	\end{enumerate}
	Define \(\xi_m^{(2)}: [\frac{1}{3}, \frac{2}{3}] \rightarrow \coprod\limits_{i=1}^{k}[0,1]_i\) by \(\xi_m^{(2)}(t) \coloneqq \xi_{m}^{(1)}(\frac{1}{3})\) for \(m = 1, 2, \ldots, r\). Finally, \(\phi_t (f,a)\) is of the form:
	\\
	\resizebox{0.86\linewidth}{!}{
		\begin{minipage}{\linewidth}
			\[
			\begin{cases}
				u^{(1)}(t) \, \diag\bigl( a(\delta_1) \otimes \mathbf{I}_{s_1}, \dots, a(\delta_l) \otimes \mathbf{I}_{s_l}, f(\xi_1^{(1)}(t)), \ldots, f(\xi_r^{(1)}(t)) \bigr) \, u^{(1)}(t)^* & \text{if } t \in \left[0, \frac{1}{3}\right], \\
				\\
				u^{(2)}(t) \, \diag\bigl( a(\delta_1) \otimes \mathbf{I}_{s_1}, \dots, a(\delta_l) \otimes \mathbf{I}_{s_l}, f(\xi_1^{(2)}(\frac{1}{3})), \ldots, f(\xi_r^{(2)}(\frac{1}{3})) \bigr) \, u^{(2)}(t)^* & \text{if } t \in \left[\frac{1}{3}, \frac{2}{3}\right], \\
				\\
				u^{(3)}(t) \, \diag\bigl( a(\delta_1) \otimes \mathbf{I}_{t_1}, \dots, a(\delta_l) \otimes \mathbf{I}_{t_l}, f(\xi_1^{(3)}(t)), \ldots, f(\xi_q^{(3)}(t)) \bigr) \, u^{(3)}(t)^* & \text{if } t \in \left[\frac{2}{3}, 1\right],
			\end{cases}
			\]
		\end{minipage}
	}\\
	
	\noindent where \( u^{(1)} \in C([0, \frac{1}{3}], M_n(\mathbb{C})) \) and \( u^{(3)} \in C([\frac{2}{3}, 1], M_n(\mathbb{C})) \) are constant functions defined by \( u^{(1)}(t) = U_0 \) for \( t \in [0, \frac{1}{3}] \) and \( u^{(3)}(t) = U_1 \) for \( t \in [\frac{2}{3}, 1] \). The function \( u^{(2)} \in C([\frac{1}{3}, \frac{2}{3}], M_n(\mathbb{C})) \) defines a continuous path of unitaries from \( U_0 \) to \( U_1 P \) for some permutation matrix \( P \). We define a piecewise continuous function \( U: [0,1] \rightarrow M_n(\mathbb{C}) \) by
	\begin{equation}\label{eqn:EndpointInvariance}
		U(t) \coloneq \begin{cases}
			u^{(1)}(t) & \text{if } t \in [0, \frac{1}{3}) \\
			u^{(2)}(t) & \text{if } t \in [\frac{1}{3}, \frac{2}{3}) \\
			u^{(3)}(t) & \text{if } t \in [\frac{2}{3}, 1].
		\end{cases}
	\end{equation}
	Note that the two forms of \(\phi_{\frac{2}{3}}\) above coincide, and there exists a permutation matrix \(Q \in M_n(\mathbb{C})\) such that 
	\begin{align*}
		\hspace*{-2cm}	\diag\bigl( a(\delta_1) \otimes \mathbf{I}_{s_1}, \dots, a(\delta_l) \otimes \mathbf{I}_{s_l}, f(\xi_1^{(2)}(\frac{1}{3})), \ldots, f(\xi_r^{(2)}(\frac{1}{3})) \bigr)=\\Q \, \diag\bigl( a(\delta_1) \otimes \mathbf{I}_{t_1}, \dots, a(\delta_l) \otimes \mathbf{I}_{t_l}, f(\xi_1^{(3)}(\frac{2}{3})), \ldots, f(\xi_q^{(3)}(\frac{2}{3})) \bigr) \, Q^*
	\end{align*}
	by construction.
\end{proof}
\begin{rmk}\label{rmk:D-map}
	Given a unital \(n\)-standard map \(\phi: A \rightarrow M_q(\mathbb{C})\) as in Definition~\ref{defn:n-standardform}, note that there exist permutation matrices \(Q_1, Q_2, \ldots, Q_n\), with \(Q_1 = \id\), such that for each \(m = 1, 2, \ldots, n - 1\), we have
	\begin{equation}
		\label{eqn:perm_diag}
		\begin{aligned}
			&Q_m \diag\left( h \circ \xi_1^{(m)}(z_{m}), \ldots, h \circ \xi_{k(m)}^{(m)}(z_{m}) \right) Q_{m}^* \\
			&= Q_{m+1} \, \diag\left( h \circ \xi_1^{(m+1)}(z_{m}), \ldots, h \circ \xi_{k(m+1)}^{(m+1)}(z_{m}) \right) Q_{m+1}^*
		\end{aligned}
	\end{equation}
	for all \(h \in A\). This follows from the fact that the spectrum \(\mathrm{Sp}(\phi_{z_i})\) remains unchanged across adjacent subintervals.
	
	Let \(R : [0,1] \rightarrow M_q(\mathbb{C})\) be defined by
	\begin{equation}
		R(t) \coloneqq
		\begin{cases}
			u^{(1)}(t) Q_1^* & \text{if } t \in [z_0, z_1), \\
			u^{(2)}(t) Q_2^* & \text{if } t \in [z_1, z_2), \\
			\quad\vdots \\
			u^{(n)}(t) Q_n^* & \text{if } t \in [z_{n-1}, z_n].
		\end{cases}
	\end{equation}
	
	Define \(\theta: A \rightarrow C([0,1], M_q(\mathbb{C}))\) by setting, for each \(t \in I_m = [z_{m-1}, z_m)\) with \(m = 1, 2, \ldots, n\) and \(h\in A\),
	\begin{equation}
		\theta_t(h) \coloneqq Q_m
		\begin{pmatrix}
			h \circ \xi_1^{(m)}(t) & & & \\
			& h \circ \xi_2^{(m)}(t) & & \\
			& & \ddots & \\
			& & & h \circ \xi_{k(m)}^{(m)}(t)
		\end{pmatrix}
		Q_m^*.
	\end{equation}
	Here, each \(\xi_i^{(m)} \in C(I_m, \mathrm{sp}(A))\). Then \( R\) is a piecewise continuous unitary-valued function and \(\theta\) is a \(^*\)-homomorphism satisfying \( \phi_t = R(t)\theta_t R(t)^* \) for all \( t \in [0,1] \).
	
	Similarly, for \( A, B \in \text{1-NCCW}_1 \) with \(B\coloneq B(E', F', \alpha_0, \alpha_1) \)and a unital \( n \)-standard map \( \phi: A \rightarrow B \), there exists a piecewise continuous unitary-valued function \( R: [0,1] \rightarrow \mathcal{U}(F') \) and a \( ^* \)-homomorphism \( \theta: A \rightarrow C([0,1], F') \) such that \( \phi_t = R(t) \theta_t R(t)^* \) for all \( t \in [0,1] \). Moreover,
	
	\begin{equation}\label{eqn:D-map}
		\theta_t(h) \coloneq Q_t \begin{pmatrix}
			h \circ \xi_1^{(m)}(t) & \\
			& h \circ \xi_2^{(m)}(t) & \\
			& & \ddots & \\
			& & & h \circ \xi_{k(m)}^{(m)}(t) \\
		\end{pmatrix} Q_t^*
	\end{equation}
	for all \( h \in A \) and all \( t \) in some dividing interval \( I_m \), where \( \xi_i^{(m)} \in C(I_m , \text{sp}(A)) \) and \( Q_t \in F' \) are permutation matrices.
	
\end{rmk}

\begin{defn} Given 1-dimensional NCCW complexes $A$ and $B$ in \(\text{1-NCCW}_1 \), and a unital $n$-standard map $\phi: A \to B$, the pair $(\theta, R)$, consisting of a $^*$-homomorphism $\theta$ and a piecewise continuous unitary-valued function $R$ as described in Remark \ref{rmk:D-map}, is called a \emph{$\mathcal{D}$-pair} associated with $\phi$.
	
\end{defn}
\begin{egg}
	Let \( V : [0,1] \rightarrow M_n(\mathbb{C}) \) be a piecewise continuous function defined by
	\[
	V(t) \coloneq \begin{cases}
		\mathbf{I}_n & \text{if } t \in \left[0, \tfrac{1}{3}\right), \\
		\mathbf{I}_n & \text{if } t \in \left[\tfrac{1}{3}, \tfrac{2}{3}\right), \\
		Q & \text{if } t \in \left[\tfrac{2}{3}, 1\right],
	\end{cases}
	\]
	where \( Q \) is the permutation matrix referenced in Lemma~\ref{lem:3-standardform}, and \( U \) is defined as in \eqref{eqn:EndpointInvariance}. Let \( R : [0,1] \rightarrow M_n(\mathbb{C}) \) be defined by \( R(t) \coloneq U(t) V(t)^* \) and \( \theta : A \rightarrow C([0,1], M_n(\mathbb{C})) \) be defined by
	\\
	\resizebox{0.87\linewidth}{!}{
		\begin{minipage}{\linewidth}
			\[
			\theta_t(f,a) \coloneqq 
			\begin{cases}
				\, \diag\bigl( a(\delta_1) \otimes \mathbf{I}_{s_1}, \dots, a(\delta_l) \otimes \mathbf{I}_{s_l}, f(\xi_1^{(1)}(t)), \ldots, f(\xi_r^{(1)}(t)) \bigr) & \text{if } t \in \left[0, \frac{1}{3}\right), \\
				\\
				\diag\bigl( a(\delta_1) \otimes \mathbf{I}_{s_1}, \dots, a(\delta_l) \otimes \mathbf{I}_{s_l}, f(\xi_1^{(2)}(\frac{1}{3})), \ldots, f(\xi_r^{(2)}(\frac{1}{3})) \bigr) & \text{if } t \in \left[\frac{1}{3}, \frac{2}{3}\right), \\
				\\
				Q \, \diag\bigl( a(\delta_1) \otimes \mathbf{I}_{t_1}, \dots, a(\delta_l) \otimes \mathbf{I}_{t_l}, f(\xi_1^{(3)}(t)), \ldots, f(\xi_q^{(3)}(t)) \bigr) \, Q^* & \text{if } t \in \left[\frac{2}{3}, 1\right]
			\end{cases}
			\]
		\end{minipage}
	}\\
	
	\noindent for all \((f,a)\in A\) and all \( t \in [0,1] \). Then \( (\theta, R) \) is a \( \mathcal{D} \)-pair associated with the 3-standard map constructed in Lemma \ref{lem:3-standardform}.
\end{egg}
Following some of the arguments presented in \cite[Theorem 3.6]{Liu19}, we prove the following result.
\begin{theo}\label{thm:ApproxByStandardMap}
	Let \( A \coloneqq A(E,F,\beta_0, \beta_1) \), \( B \coloneqq B(E', F', \alpha_0, \alpha_1) \in \text{1-NCCW}_1\), \(\mathcal{F}\) be a finite set, and \(\epsilon > 0\). Let \(\phi : A \rightarrow B\) be a unital \(^*\)-homomorphism. Then there exists a unital \(n\)-standard map \(\psi: A \rightarrow B\) such that \(\Vert \phi(h) - \psi(h) \Vert < \epsilon \) for all \( h \in \mathcal{F} \). If \(\phi\) is injective, then \(\psi\) is also injective.
\end{theo}
\begin{proof}
	Without loss of generality, we assume \( 0 < \epsilon < 1 \). For each \( i \), \( 1 \le i \le k' \), apply Lemma \ref{lem:3-standardform} to \( A \), \( \frac{\epsilon}{2} \), \( M_{f_i'}(\mathbb{C}) \), and the finite set \( \mathcal{F} \) to obtain \( \eta_1^{(i)} \), \( \eta^{(i)} \), and \( \delta^{(i)} \). Set \[
	\eta_1 \coloneq \min\{ \eta_1^{(i)} : 1 \le i \le k' \} \quad \text{and} \quad 	H_i \coloneq H(\eta^{(i)}) \cup \widetilde{H}(\eta^{(i)}) \cup H(\eta_1^{(i)}) \cup H(\eta_1),
	\]  
	for each \( i \).
	Then there exists \( \overline{\delta}_i > 0 \) such that, for any \( w, w' \in [0,1]_i \) with \( \operatorname{dist}(w, w') < \overline{\delta}_i \), we have 
	\begin{equation} \label{eqn:ApproxInequality}
		\left\| \phi_w(h) - \phi_{w'}(h) \right\| 
		< \min\left\{ \frac{\epsilon}{2}, \frac{\delta^{(i)}}{8} \right\} 
		\quad \text{for all } 
		h \in H_i .
	\end{equation}
	Set \(\delta \coloneq \min\{ \overline{\delta}_i : 1 \le i \le k' \}\) For each \(i\), consider an \(m\)-partition \(0 = z_0 < z_1 < \ldots < z_m = 1\) of \([0,1]_i\) with \(\text{dist}(z_{j-1}, z_j) =\frac{ \delta}{2}\) for \(j = 1, 2, \ldots, m\). Denote the dividing interval \(I_j^i \coloneqq [z_{j-1}, z_j]_i\) for \(j = 1, 2, \ldots, m\) and \(i = 1, 2, \ldots, k'\).
	
	
	For a fixed \(i \in \{1, \ldots, k'\}\), set \(w_{j-1} \coloneqq (z_{j-1}, i)\) for each \(j = 1, 2, \ldots, m\). Let \(\phi_{w_0}\) be as in \eqref{eqn:InitialPath} and \(\phi_{w_1}\) be as in \eqref{eqn:FinalPath}. By \eqref{eqn:ApproxInequality} and Lemma \ref{lem:3-standardform}, we obtain a \(3\)-standard map \(\psi^{i1} \coloneqq \left.\psi \right|_{I_1^i}: A \rightarrow C(I_1^i, M_{f_i'}(\mathbb{C}))\) satisfying
	\begin{equation}\label{eqn:InitialApproxStep}
		\Vert \psi^{i1}_z(f) - \phi_{w_0}(f) \Vert < \frac{\epsilon}{2} \,\,\, \forall f \in \mathcal{F}, \,\, z \in I_1^i,
	\end{equation}
	with \(\psi^{i1}_{z_0} \coloneqq \phi_{w_0}\), \(\psi^{i1}_{z_1} \coloneqq \phi_{w_1}\), and for each \(\mu \in (\text{Sp}(\psi^{i1}_{z_0}) \cup \text{Sp}(\psi^{i1}_{z_1})) \cap \coprod\limits_{q=1}^{k} (0,1)_q\), we have
	\begin{equation} 
		\overline{B_{4\eta_1^{(i)}}(\mu)} \subset \bigcup_{t \in I_1^i} \text{Sp}(\psi^{i1}_t). 
	\end{equation}
	
	By repeating the process for each \(j = 2, \ldots, m\), we similarly obtain, from Lemma \ref{lem:3-standardform}, \(3\)-standard maps \(\psi^{ij} \coloneqq \left.\psi \right|_{I_j^i}: A \rightarrow C(I_j^i, M_{f_i'}(\mathbb{C}))\) satisfying
	\begin{enumerate}
		\item \(\psi^{ij}_{z_{j-1}} = \phi_{w_{j-1}}\) and \(\psi^{ij}_{z_j} = \phi_{w_j}\) for each \(j = 2, \ldots, m\),
		\item For each \(j = 2, \ldots, m\), we have that 
		\begin{equation}\label{eqn:3m-standardInequality}
			\Vert \psi^{ij}_z(f) - \phi_{w_{j-1}}(f) \Vert < \frac{\epsilon}{2} \,\, \forall f \in \mathcal{F}, \,\, z \in I_j^i.
		\end{equation}
		\item For each \(j = 2, \ldots, m\) and each \(\mu \in (\text{Sp}(\psi^{ij}_{z_{j-1}}) \cup \text{Sp}(\psi^{ij}_{z_j})) \cap \coprod\limits_{q=1}^{k} (0,1)_q\), we have 
		\begin{equation}\label{eqn:BallinSpectrum} 
			\overline{B_{4\eta_1^{(i)}}(\mu)} \subset \bigcup_{t \in I_j^i} Sp(\psi^{ij}_t).
		\end{equation}
	\end{enumerate}
	It follows from \eqref{eqn:ApproxInequality}, \eqref{eqn:InitialApproxStep}, and \eqref{eqn:3m-standardInequality} that
	\begin{equation}\label{eqn:closerestriction}
		\Vert \left.\psi \right|_{I_j^i}(f) - \left.\phi \right|_{I_j^i}(f) \Vert < \epsilon
	\end{equation}
	for all \(f \in \mathcal{F}\) and all \(j = 1, 2, \ldots, m\). We can fit all the \(\psi^{ij} \coloneqq \left.\psi \right|_{I_j^i}\) together to define a \(3m\)-standard map \(\psi^i : A \rightarrow C([0,1], M_{f_i'}(\mathbb{C}))\). For each \(\delta_r \in \text{Sp}(E')\) and \(h \in A\), define
	\begin{equation}\label{eqn:Equalityof3m-standard}
		\psi_{\delta_r}(h) \coloneqq \phi_{\delta_r}(h) \,\,\, \text{for all}\,\, h \in A,
	\end{equation}
	where \(\phi_{\delta_r} \coloneqq \delta_r \circ \phi\). Using \eqref{eqn:Equalityof3m-standard} and combining all the \(\psi^i\), we define a \(3m\)-standard map \(\psi : A \rightarrow B\), since \(\phi_{\tau}(h) = \psi_{\tau}(h)\) for all \(\tau = 0, 1\) and \(h \in A\). We obtain from \eqref{eqn:closerestriction} that
	\begin{equation}
		\Vert \psi(f) - \phi(f) \Vert < \epsilon
	\end{equation}
	for all \(f \in \mathcal{F}\).
	
	Next, we prove that if \(\phi\) is injective, meaning that \(\text{Sp}(\phi) \coloneqq \bigcup\limits_{\mu \in \text{Sp}(B)} \text{Sp}(\phi_\mu) = \text{Sp}(A)\), then the map \(\psi\) defined above is also injective. For a fixed \(1\le r\le k\), let \(\upsilon_0 \in (0,1)_r\). Set \(\mathcal{V} \coloneqq \{\upsilon \in (0,1)_r \mid \text{dist}(\upsilon, \upsilon_0) \le \eta_1\}\) and \(Y \coloneqq \coprod\limits_{i=1}^{k'} (0,1)_i\). The injectivity of \(\phi\) implies that
	\[
	\mathcal{V} \subset \text{Sp}(\phi) \coloneqq \bigcup\limits_{\mu \in Y} \text{Sp}(\phi_\mu) \cup \bigcup\limits_{\mu \in Sp(E')} \text{Sp}(\phi_\mu).
	\]
	Hence, \(\mathcal{V} \cap \bigcup\limits_{\mu \in Y} \text{Sp}(\phi_{\mu})\) contains infinitely many points. Choose \(\upsilon_1 \in \mathcal{V} \cap \bigcup\limits_{\mu \in Y} \text{Sp}(\phi_{\mu}) \cap [\eta_1 , 1 - \eta_1]_r\). Then \(\upsilon_1 \in \text{Sp}(\phi_z)\) for some \(z \in I_j^i\). By \eqref{eqn:ApproxInequality}, \(\Vert \phi_z(h) - \phi_{w_{j-1}}(h) \Vert < 1\) for all \(h \in H(\eta_1 )\). It follows from Lemma \ref{lem:Pairing} that there exists \(\upsilon_2 \in \text{Sp}(\phi_{w_{j-1}})\) such that \(\text{dist}(\upsilon_1, \upsilon_2) < 2 \eta_1\). Consequently, we get that
	\[
	d(\upsilon_0, \upsilon_2) \le d(\upsilon_0, \upsilon_1) + d(\upsilon_1, \upsilon_2) < \eta_1 + 2\eta_1 \le \eta_1 ^{(i)}+ 2\eta_1^{(i)}  = 3\eta_1^{(i)},
	\]
	and \(\upsilon_0 \in \overline{B_{4\eta_1^{(i)}}(\upsilon_2)} \subset \bigcup\limits_{t \in I_j^i} \text{Sp}(\psi_t)\) by \eqref{eqn:BallinSpectrum}. Since \(\upsilon_0\) and \(r\) are arbitrary, we have \(\coprod\limits_{i=1}^{k}(0,1)_i \subset \text{Sp}(\psi)\).

	For \(\upsilon_0 \in (0, \eta_1)_r\), we have
	\[
	d(0, \upsilon_0) \le d(0, \upsilon_0) + d(\upsilon_0, \upsilon_2) \le 4\eta_1^{(i)}.
	\]
	Similarly, for \(\upsilon_0 \in (1 - \eta_1, 1)_r\),
	\[
	d(1, \upsilon_0) \le d(1, \upsilon_0) + d(\upsilon_0, \upsilon_2) \le 4\eta_1^{(i)}.
	\]
	Hence, \((0,r)\) and \((1,r)\) are in \(\overline{B_{4\eta_1^{(i)}}(\upsilon_2)} \subset \bigcup\limits_{t \in I_j^i} \text{Sp}(\psi_t)\). Thus, \((0,r)\) and \((1,r)\) are in \(\text{Sp}(\psi)\) for all \(r = 1, 2, \ldots, k\). We conclude that \(\text{Sp}(E) \subset \text{Sp}(\psi)\) and \(\text{Sp}(\psi) = \text{Sp}(A)\). Therefore, \(\psi\) is injective.
\end{proof}
\section{$C^*$-diagonals of Inductive Limits of $1$-dimensional NCCW complexes}
Finally, we establish the existence of \( C^* \)-diagonals in certain inductive limits of 1-dimensional NCCW complexes. Here, we denote the set of piecewise continuous function from \([0,1]\) to a finite dimensional \(C^*\)-algebra \(F\) by \(PC([0,1],F)\). For a piecewise continuous function \( R: [0,1] \rightarrow \mathcal{U}(F') \), we define a piecewise continuous function \( R^*: [0,1] \rightarrow \mathcal{U}(F') \) by \( W^*(t) = W(t)^* \).

\begin{lem}\label{lem:ApproxUnitaryEqui}
	Let \(A\) and \(B\) be 1-dimensional NCCW complexes with \(K_1(A) = 0\), and let \(\phi, \psi: A \rightarrow B\) be \( ^* \)-homomorphisms. If \(\pi \circ \phi\) and \(\pi \circ \psi\) are unitarily equivalent for every \(\pi \in \text{Sp}(B)\), then \(\phi\) and \(\psi\) are approximately unitarily equivalent. 
\end{lem}
\begin{proof}
	By \cite[Theorems 1.0.1 \& 3.2.2]{Robert12}, it suffices to show that \( Cu(\phi) = Cu(\psi) \). Let \( B \coloneq B(E,F,\beta_0, \beta_1) \). For any \( h \in A_+ \) and \( \pi \in \text{Sp}(B) \), \( \pi \circ \phi(h) \) and \( \pi \circ \psi(h) \) have the same rank since \( \pi \circ \phi \) is unitarily equivalent to \( \pi \circ \psi \). Hence, by a slight abuse of notation, we can write the Cuntz semigroup elements as \( [\pi_t \circ \phi (h)] = [\pi_t \circ \psi (h)] \) for all \( t \in [0,1] \), and \( [ \bigoplus\limits_{\pi \in \text{Sp}(E)} \pi \circ \phi(h) ] = [ \bigoplus\limits_{\pi \in \text{Sp}(E)} \pi \circ \psi(h) ] \). Therefore, by the computation of \( Cu(B) \) (see Theorem \ref{thm:CuntzCalculation}), we obtain that \( Cu(\phi) = Cu(\psi) \), and the result follows
\end{proof}
Next, we extend \cite[Theorem 3]{Santiago10} using a similar line of reasoning.
\begin{theo}\label{thm:ApproxoByGoodStandardMap}
	Suppose \( A \coloneq A(E, F, \beta_0, \beta_1) \) and \( B \coloneq B(E', F', \alpha_0, \alpha_1) \) are elements of \( \mathrm{1\text{-}NCCW}_1 \) with trivial \( K_1 \)-groups, and \( \phi : A \rightarrow B \) is a unital \( n \)-standard map. Let \( (\theta, R) \) be a \( \mathcal{D} \)-pair associated with \( \phi \). Then for any \(\epsilon > 0\) and finite set \(\mathcal{F} \subset A\), there exists a unitary \( H \in B \) and a unital \( n \)-standard map \(\psi : A \rightarrow B\) such that:
	
	\begin{enumerate}
		\item \begin{equation}
			\psi_t \coloneq \mathcal{W}(t) \theta_t \mathcal{W}(t)^*,
		\end{equation}
		where \(\mathcal{W} \in C([0,1], F')\) such that \(\mathcal{W}(0)\) and \(\mathcal{W}(1)\) are permutation matrices, and
		\item \begin{equation}
			\Vert \phi(f) - H\psi(f)H^* \Vert < \epsilon \quad \text{for all } f \in \mathcal{F}.
		\end{equation}
	\end{enumerate}
\end{theo}
\begin{proof}
	Let \(\delta_i' \in \text{Sp}(B) \cap \text{Sp}(E')\) and \((f,a) \in A\) be arbitrary. Then there exist unitaries \(V_i \in M_{e_i'}(\mathbb{C})\) such that
	\begin{equation} \label{eqn: componentdecomposition}
		\delta_i' \circ \phi(f,a) = V_i \, \diag\left( a(\delta_1) \otimes \mathbf{I}_{t_1^i}, \dots, a(\delta_l) \otimes \mathbf{I}_{t_l^i}, f(x_1^i), f(x_2^i), \ldots, f(x_{r^i}^i) \right) V_i^*.
	\end{equation}
	Using \eqref{InjectiveIdentification}, we can identify \( \phi(f,a) \) with \( \left(\phi(f,a), \bigoplus\limits_{i=1}^{l'} c_i \right) \), where \( c_i \coloneq \delta_i' \circ \phi(f,a) \).
	Note that
	\begin{eqnarray}
		\phi_0 (f,a) &=& R(0) \theta_0 (f,a) R^*(0), \label{eqn:zeroform1} \\
		&=& \alpha_0 \left( \bigoplus_{i=1}^{l'} c_i \right), \\
		&=& \left( \bigoplus_{i=1}^{l'} (c_i \otimes \mathbf{I}_{r_i}) \right) \quad \text{by \eqref{eqn:A2}}, \label{eqn:zeroform2}
	\end{eqnarray}
	Substitute the form of \(\theta_0\) from \eqref{eqn:D-map} and \(c_i\) from \eqref{eqn: componentdecomposition} into \eqref{eqn:zeroform1} and \eqref{eqn:zeroform2}, respectively. Since the spectrum of \(\phi_0\) in \eqref{eqn:zeroform1} and \eqref{eqn:zeroform2} must coincide, there exists a permutation matrix \(S_0 \in F'\) such that
	\begin{equation}\label{eqn:NCCWcondition1}
		S_0 \theta_0(f,a) S_0^* =  \left(\bigoplus_{i=1}^{l'} (V_i^* c_i V_i \otimes \mathbf{I}_{r_i})\right).
	\end{equation}
	Similarly, there exists a permutation matrix \( S_1 \in F' \) such that
	\begin{equation}\label{eqn:NCCWcondition2}
		S_1 \theta_1(f,a) S_1^* =  \left(\bigoplus_{i=1}^{l'} (V_i^* c_i V_i \otimes \mathbf{I}_{s_i})\right)
	\end{equation}
	Let \(\mathcal{W} \in C([0,1], F')\) define a continuous path of unitaries in \(F'\) connecting \(S_0\) to \(S_1\). 
	The \(^*\)-homomorphisms \(A \rightarrow C([0,1], F')\) that maps \(h \in A\) to \(\mathcal{W}\theta(h)\mathcal{W}^*\) and \(A \rightarrow E'\) that maps \(h \in A\) to \(\bigoplus\limits_{i=1}^{l'} V_i^* (\delta_i' \circ \phi(h)) V_i\) define a map \(\psi: A \rightarrow B\) satisfying
	\begin{equation}
		\psi_t \coloneq \mathcal{W}(t) \theta_t \mathcal{W}(t)^* \,\, \forall \,\,t\in [0,1].
	\end{equation}
	\(\psi(h) \in B\). \(\psi\) is a well-defined \( ^*\)-homomorphism since
	\[
	\alpha_0 \left(\bigoplus\limits_{i=1}^{l'} V_i^* c_i V_i\right) = S_0 \theta_0(h) S_0^* \quad \text{and} \quad \alpha_1 \left(\bigoplus\limits_{i=1}^{l'} V_i^* c_i V_i\right) = S_1 \theta_1(h) S_1^*
	\]
	for all \(h \in A\). By the definitions of \(\phi\) and \(\psi\), \(\pi \circ \psi\) is unitarily equivalent to \(\pi \circ \phi\) for all \(\pi \in \text{Sp}(B)\). The second part of the theorem follows from Lemma \ref{lem:ApproxUnitaryEqui}.
\end{proof}

An application of Elliott's intertwining argument gives the following:
\begin{lem}[{\cite[Corollary 2.3.3]{RS02}}]\label{lem:Isomorphism}Let \( A \coloneq \lim\limits_{\rightarrow} (A_n, \phi_n) \) and \( B \coloneq \lim\limits_{\rightarrow} (A_n, \psi_n) \) be inductive limit \( C^* \)-algebras, where each \( A_n \) is a unital \( C^* \)-algebra and the connecting maps \(\phi_n\) and \(\psi_n\) are unital \( ^* \)-homomorphisms. If each \( A_n \) is separable, and \(\phi_n\) is approximately unitarily equivalent to \(\psi_n\) for every \( n \), then \( A \cong B \). 
\end{lem}
\begin{theo}[{\cite[Theorem 1.10]{Li20}}]\label{thm:InductiveLimitofDiagonals}
	For each \( n \), let \( B_n \) be a \( C^* \)-diagonal of a \( C^* \)-algebra \( A_n \), \( N_n \coloneq N_{A_n}(B_n) \) the normalizer of \( B_n \) in \( A_n \), and \( P_n : A_n \rightarrow B_n \) the faithful conditional expectation of \( A_n \) onto \( B_n \). If \( \phi_n : A_n \rightarrow A_{n+1} \) is an injective \( ^* \)-homomorphism satisfying \( \phi_n(B_n)\subset B_{n+1} \), \( \phi_n (N_n)\subset N_{n+1} \), and \( P_{n+1}\circ \phi_n = \phi_n \circ P_n \) for all \( n \), then \( \lim\limits_{\rightarrow} (B_n, \phi_n) \) is a \( C^* \)-diagonal of \( \lim\limits_{\rightarrow} (A_n, \phi_n) \).
\end{theo}
We are now ready to prove the main result of this paper.
\begin{theo}\label{thm:Mainresult}
	Let $A \coloneq \lim\limits_{\rightarrow} (A_n, \phi_n)$ be a unital, separable $C^*$-algebra, where each $A_n \in \text{1-NCCW}_1$ is a 1-dimensional NCCW complex with trivial $K_1$-groups, and the connecting maps $\phi_n$ are unital and injective. Then $A$ has a $C^*$-diagonal.
\end{theo}
\begin{proof}
	Let \( A_n \coloneq A_n(E_n, F_n, \beta_{0}^n, \beta_{1}^n) \). For any \( \epsilon > 0 \), \( n \), and finite set \( \mathcal{F} \subset A_n \), it follows from Theorem \ref{thm:ApproxByStandardMap} that there exists an injective unital \( m_n \)-standard map \( \zeta_n : A_n \rightarrow A_{n+1} \) such that
	\begin{equation}\label{eqn:ApproxByStandard}
		\Vert \phi_n (a) - \zeta_n (a)\Vert < \frac{\epsilon}{2}\,,\,\,\, \text{for all}\,\,\,\, a \in \mathcal{F}.
	\end{equation}
	Let \( (\theta_n, R_n) \) be a \( \mathcal{D} \)-pair associated with \( \zeta_n \). By Theorem \ref{thm:ApproxoByGoodStandardMap}, there exists an injective unital \( m_n \)-standard map \( \psi_n: A_n \rightarrow A_{n+1} \) described by \( \mathcal{W}_{n+1}\theta_n \mathcal{W}_{n+1}^{*}: A_n \rightarrow C([0,1], F_{n+1}) \) and \( \gamma_n \coloneq \pi_e \circ \psi_n : A_n \rightarrow E_{n+1} \)
	, where \( \mathcal{W}_{n+1} \in C([0,1], F_{n+1}) \) with \( \mathcal{W}_{n+1}(0) \) and \( \mathcal{W}_{n+1}(1) \) being permutation matrices. Moreover,
	\begin{equation}\label{eqn:ApproxByGoodStandard}
		\Vert \zeta_n (a) - H_n \psi_n (a) H_n^* \Vert < \frac{\epsilon}{2}
	\end{equation}
	for some unitaries \( H_n \in A_{n+1} \) and all \(a\in \mathcal{F}\). Combining \eqref{eqn:ApproxByStandard} and \eqref{eqn:ApproxByGoodStandard}, we conclude that \( \Vert \phi_n (a) - H_n \psi_n (a) H_n^* \Vert < \epsilon \) for all \( a \in \mathcal{F} \). Hence, \( \phi_n \) is approximately unitarily equivalent to an \( m_n \)-standard map \(\psi_n\). Since 1-dimensional NCCW complexes are finitely generated and separable \cite[Lemma 2.3]{Eilers98}, we have that \( A \cong \lim\limits_{\rightarrow} (A_n , \psi_n) \) by Lemma \ref{lem:Isomorphism}.
	
	We adopt a slight modification of the technique described in \cite[Remark 4.1]{Li20} to further transform \( A \) into an inductive limit \( C^* \)-algebra with connecting \( ^* \)-homomorphisms that preserve the \( C^* \)-diagonals of each building block.
	
	
	Regard \( \theta_n : C([0,1],F_n) \oplus E_n \rightarrow PC([0,1],F_{n+1}) \) as an extension of \( \theta_n : A_n \rightarrow C([0,1],F_{n+1}) \), where the dividing interval is considered as \( [z_i, z_{i+1}) \) instead of \( [z_i, z_{i+1}] \). Then \( \theta \) is a \( ^* \)-algebra homomorphism, and a \( C^* \)-homomorphism on each dividing interval. For each \( n \) and \( s \in [0,1] \), let
	\begin{enumerate}
		\item \(\Omega_n(s) \subset (0,1)\) be the finite set \( \text{Sp}(\pi_s \circ \psi_n) \cap \coprod\limits_{i=1}^{k}(0,1)_i \), ignoring multiplicity and index.
		\item \(\Sigma_n \subset (0,1)\) be the finite set \( \text{Sp}(\pi_e \circ \psi_n) \cap \coprod\limits_{i=1}^{k}(0,1)_i \), ignoring multiplicity and index.
	\end{enumerate}
	We proceed by selecting a unitary \( Z_{n+1} \in \mathcal{U}(C([0,1], F_{n+1})) \), a unitary \( \mathsf{V}_n \in \mathcal{U}(C([0,1], F_n)) \), and a piecewise continuous function of permutation matrices \( C_{n+1} \in PC([0,1], F_{n+1}) \) such that, for all \( n \), \( \mathsf{V}_n(s) = 1 \) for all \( s \in (\Omega_n(0) \cup \Omega_n(1) \cup \Sigma_n) \), \( \mathsf{V}_n(t) \) is a permutation matrix for each \( t \in \{0,1\} \), \( Z_{n+1}(t) = \mathcal{W}_{n+1}(t) \) for \( t \in \{0, 1\} \), and
	\[
	\mathsf{V}_{n+1} \coloneq C_{n+1} \theta_n(\mathsf{V}_n, \I_n) Z_{n+1}^{*},
	\]
	as follows.
	
	Define \( \mathsf{V}_1 := 1 \). If \( \mathsf{V}_n \) and \( Z_n \) have been chosen, then \( \theta_n(\mathsf{V}_n, \I_n) \) may not be in \( \mathcal{U}(C([0,1], F_{n+1})) \). However, \( \mathsf{V}_n(t) \) being a permutation matrix for each \( t = 0,1 \) ensures the existence of a piecewise continuous function of permutation matrices \( C_{n+1} \in PC([0,1],F_{n+1}) \) such that \( C_{n+1}\theta_n(\mathsf{V}_n, \I_n) \) is unitary. Pick \( Z_{n+1} \in \mathcal{U}(C([0,1], F_{n+1})) \) such that \( Z_{n+1}(t) \coloneq \mathcal{W}_{n+1}(t) \) for \( t \in \{0,1\} \), and
	\[
	Z_{n+1}(s) \coloneq C_{n+1}(s) \theta_n(\mathsf{V}_n, \I_n)(s) \quad \forall \, s \in (\Omega_n(0) \cup \Omega_n(1) \cup \Sigma_n).
	\]
	Moreover, set  
	\begin{equation}\label{eqn:PermutationConfirmation}
		\mathsf{V}_{n+1} \coloneq C_{n+1} \theta_n(\mathsf{V}_n, \I_n) Z_{n+1}^{*}
	\end{equation}
	and
	\begin{equation}
		\overline{A}_n \coloneq A(E_n, F_n, \alpha_0^n, \alpha_1^n),
	\end{equation}
	where \( \alpha_t^n \coloneq \mathsf{V}_n(t) \beta_t^n \mathsf{V}_n^{*}(t) \), and \( t = \{0,1\} \). Let \( \overline{\psi}_n: A_n \rightarrow A_{n+1} \) be a \( ^* \)-homomorphism defined by the \( ^* \)-homomorphisms \( \gamma_n \) and \( Z_{n+1} \theta_n Z_{n+1}^{*}: A_n \rightarrow C([0,1], F_{n+1}) \in C([0,1], F_{n+1})  \). 
	
	If \( (f,a) \in \overline{A}_n \), then \( (\mathsf{V}_n^* f \mathsf{V}_n, a) \in A_n \). Moreover, if \( (f,a) \in \overline{A}_n \), we have
	\begin{equation}
		C_{n+1} \theta_n(f,a) C_{n+1}^* 
		= C_{n+1} \theta_n(\mathsf{V}_n, \I_n) \, \theta_n(\mathsf{V}_n^* f \mathsf{V}_n, a) \, \theta_n(\mathsf{V}_n^*, \I_n) C_{n+1}^*
	\end{equation}
	which lies in \( C([0,1], F_{n+1}) \), since \( \theta_n \) is a \( ^* \)-algebra homomorphism, and both \( C_{n+1} \theta_n(\mathsf{V}_n, \I_n) \) and \( \theta_n(\mathsf{V}_n^* f \mathsf{V}_n, a) \) are elements of \( C([0,1], F_{n+1}) \). 
	
	Consequently, the map
	\[
	\overline{\theta}_n \coloneq C_{n+1} \theta_n C_{n+1}^* : \overline{A}_n \rightarrow C([0,1], F_{n+1})
	\]
	is a \( ^* \)-homomorphism.
	
	Let \( \widehat{\psi}_n : \overline{A}_n \rightarrow \overline{A}_{n+1} \) be the \( ^* \)-homomorphism defined by the maps \( \overline{\theta}_n \) and \( \overline{\gamma}_n : \overline{A}_n \rightarrow E_{n+1} \), where \( \overline{\gamma}_n(f, a) = \gamma_n(\mathsf{V}_n^* f \mathsf{V}_n, a) \).
	
	
	

	Using \eqref{eqn:PermutationConfirmation}, we obtain the following commutative diagram:
	\begin{center}

		\tikzset{every picture/.style={line width=0.75pt}} 
		
		\begin{tikzpicture}[x=0.75pt,y=0.75pt,yscale=-1,xscale=1]
			
			\draw    (301.67,97.33) -- (383.67,98.31) ;
			\draw [shift={(385.67,98.33)}, rotate = 180.68] [color={rgb, 255:red, 0; green, 0; blue, 0 }  ][line width=0.75]    (8.74,-2.63) .. controls (5.56,-1.12) and (2.65,-0.24) .. (0,0) .. controls (2.65,0.24) and (5.56,1.12) .. (8.74,2.63)   ;
			\draw    (280.09,113.64) -- (280.2,160.6) ;
			\draw [shift={(280.2,162.6)}, rotate = 269.87] [color={rgb, 255:red, 0; green, 0; blue, 0 }  ][line width=0.75]    (8.74,-2.63) .. controls (5.56,-1.12) and (2.65,-0.24) .. (0,0) .. controls (2.65,0.24) and (5.56,1.12) .. (8.74,2.63)   ;
			\draw    (290,178.6) -- (383,178.6) ;
			\draw [shift={(385,178.6)}, rotate = 180] [color={rgb, 255:red, 0; green, 0; blue, 0 }  ][line width=0.75]    (8.74,-2.63) .. controls (5.56,-1.12) and (2.65,-0.24) .. (0,0) .. controls (2.65,0.24) and (5.56,1.12) .. (8.74,2.63)   ;
			\draw    (407.09,113.64) -- (407.2,160.6) ;
			\draw [shift={(407.2,162.6)}, rotate = 269.87] [color={rgb, 255:red, 0; green, 0; blue, 0 }  ][line width=0.75]    (8.74,-2.63) .. controls (5.56,-1.12) and (2.65,-0.24) .. (0,0) .. controls (2.65,0.24) and (5.56,1.12) .. (8.74,2.63)   ;
			
			\draw (122,147) node [anchor=north west][inner sep=0.75pt]   [align=left] {$ $};
			\draw (390.98,86.48) node [anchor=north west][inner sep=0.75pt]    {$A_{n}{}_{+}{}_{1}$};
			\draw (268.64,168.68) node [anchor=north west][inner sep=0.75pt]    {$\overline{A}_{n}$};
			\draw (391.27,170.92) node [anchor=north west][inner sep=0.75pt]    {$\overline{A}_{n+1}$};
			\draw (271.96,88.97) node [anchor=north west][inner sep=0.75pt]    {$A_{n}$};
			\draw (183.41,129.5) node [anchor=north west][inner sep=0.75pt]    {$Ad( 	\mathsf{V}_{n}) \oplus \id\ $};
			\draw (415.27,126.92) node [anchor=north west][inner sep=0.75pt]    {$Ad( 	\mathsf{V}_{n+1}) \oplus \id\ $};
			\draw (333.72,187.54) node [anchor=north west][inner sep=0.75pt]    {$\hat{\psi }_{n}$};
			\draw (324.44,65.6) node [anchor=north west][inner sep=0.75pt]    {$\overline{\psi }_{n}$};

		\end{tikzpicture}
	\end{center}
	
	Notice that \(\overline{A}_n\) satisfies conditions (A1) and (A2). Using a similar argument as in the beginning of the proof and the commutative diagram above, we conclude that 
	\[
	A \cong \lim_{\rightarrow}(A_n, \psi_n) \cong \lim_{\rightarrow}(A_n, \overline{\psi}_n) \cong \lim_{\rightarrow}(\overline{A}_n, \widehat{\psi}_n).
	\]
	Now, we show the existence of a \(C^*\)-diagonal of \(A\). Let 
	\[
	\widehat{B}_n \coloneq \{ (f,a) \in \widehat{A}_n : f(t) \in DF_n \,\, \forall \,\, t \in [0,1] \},
	\]
	where \(DF_n\) is the set of diagonal matrices of \(F_n\). Recall that \(\widehat{B}_n\) is a \(C^*\)-diagonal of \(\overline{A}_n\) for all \(n\) by Proposition \ref{prop:C*-diagonal}. Using the definition of \(\widehat{\psi}_n\) and \(\theta_n\) (see \eqref{eqn:D-map}), it is evident that
	\begin{itemize}
		\item \(\widehat{\psi}_n (\widehat{B}_n) \subset \widehat{B}_{n+1}\) for all \(n\).
		\item \(\widehat{\psi}_n (N_n) \subset N_{n+1}\) since the normalizers of \(\widehat{B}_n\) are of the form \((f,a)\), where \(f(t)\) is a generalized permutation matrix.
		\item \(P_{n+1} \circ \widehat{\psi}_n = \widehat{\psi}_n \circ P_n\), where \(P_n\) is the canonical conditional expectation.
	\end{itemize}
	By Theorem \ref{thm:InductiveLimitofDiagonals}, we conclude that \(\lim\limits_{\rightarrow}(\widehat{B}_n, \widehat{\psi}_n)\) is a \(C^*\)-diagonal of \(A\).

	
\end{proof} 
\begin{cor}\label{cor:InductiveLimit}
	Let $A \coloneq \lim\limits_{\rightarrow} (A_n, \phi_n)$ be a unital, separable $C^*$-algebra, where each $A_n \in \text{1-NCCW}_1$ is a 1-dimensional NCCW complex with trivial $K_1$-groups, and the connecting maps $\phi_n$ are unital and injective. Then there exists an inductive system $(B_k, \varphi_k)$, where each $B_k \in \mathrm{1\text{-}NCCW}_1$, and each connecting map $\varphi_k: B_k \to B_{k+1}$ is unital, injective, and $n_k$-standard for some $n_k$, such that
	\(
	A \cong \varinjlim (B_k, \varphi_k)\).
\end{cor}
\begin{proof}
	This follows from the proof of Theorem \ref{thm:Mainresult}.
\end{proof}
\begin{cor}[{cf.\ \cite[Corollary 3.6]{Li+Raad23}}]
	Let $A = \varinjlim (A_n, \phi_n)$ be a unital $AH$-algebra, where the algebras $A_n$ are of the form $\bigoplus\limits_s C(X_s^{(n)}, M_{n_s})$ with the spaces $X_s^{(n)}$ as trees, and the maps $\phi_n$ are unital injective *-homomorphisms. Then $A$ has a C*-diagonal.
\end{cor}
\begin{proof}
	Each $A_n$ is a $1$-dimensional NCCW complex with trivial $K_1$-group. Therefore, Theorem \ref{thm:Mainresult} applies.
\end{proof}
\bibliographystyle{plain}
\bibliography{mybib}
\end{document}